\documentclass[11pt,leqno]{article}
\usepackage{amsmath,amsthm}
\usepackage{amsfonts}
\usepackage{mathrsfs}
\usepackage{hyperref}
\usepackage{textcomp}
\usepackage{amssymb}
\usepackage{epsfig}
\usepackage{graphicx}
\usepackage{caption}
\usepackage{bm}
\usepackage{color} 

\numberwithin{equation}{section} 
\newtheorem{theorem}{\bf Theorem}[section]
\newtheorem{example}{\bf Example}[section]

\newtheorem{remark}{\bf Remark}[section]
\newtheorem{lemma}{\bf Lemma}[section]

\newcommand{\bphi}{\mbox{\boldmath $\phi$}}

 \newcommand{\norm}[1]{\left\lVert #1\right\rVert}

\newsavebox{\savepar}		 

\begin{document}
\title{Global Stabilization of  Two Dimensional Viscous Burgers' Equation by Nonlinear Neumann Boundary Feedback Control and its Finite Element Analysis}
\author{ Sudeep Kundu\footnote{
		Institute of Mathematics and Scientific Computing, University of Graz,
		Heinrichstr. 36, A-8010 Graz, Austria,
		Email:sudeep.kundu@uni-graz.at}\quad		
	and	
	Amiya Kumar Pani\footnote{
		Department of Mathematics, 	IIT Bombay, 
		Powai, Mumbai-400076, India,
		Email:{akp@math.iitb.ac.in}}.}
\maketitle
\abstract{ 
In this article, global stabilization results for the two dimensional (2D)	viscous Burgers' equation, that is, convergence of unsteady solution to its constant steady state solution with any initial data,  are established using a nonlinear Neumann boundary feedback control law.
Then, applying $C^0$-conforming finite element method in spatial direction, optimal error estimates in $L^\infty(L^2)$ and in $L^\infty(H^1)$- norms  for the state variable
and convergence result for the boundary feedback control law are derived. All the results preserve exponential stabilization property. 
Finally, several  numerical experiments  are conducted to confirm our theoretical findings.}

Keywords: 2D-viscous Burgers' equation, boundary feedback control, constant steady state, stabilization, finite element method, error estimate, numerical experiments\\
AMS subject classification: 35B37, 65M60, 65M15, 93B52, 93D15
\section{Introduction}
 We consider the following Neumann boundary control problem for the two-dimensional viscous Burgers' or Bateman-Burgers equation 
: seek $u=u(x,t),$ $t>0$ which satisfies
\begin{align}
&u_t-\nu\Delta u+u(\nabla u\cdot {\bf{1}})=0\qquad\text{in}\quad (x,t)\in \Omega\times(0,\infty)\label{feq1.1},\\
&\frac{\partial u}{\partial n}(x,t)=v_2(x,t)\qquad \text{on} \quad(x,t)\in \partial \Omega \times (0,\infty)\label{feq1.2},\\
&u(x,0)=u_0(x)\qquad \text{in}\quad x\in \Omega\label{feq1.3},
\end{align}
where  $\Omega\subset \mathbb{R}^2$ is a bounded domain with smooth boundary $\partial \Omega,$   $\nu>0$ is the diffusion constant with the diffusion term $\nu\Delta u$, $v_2$ is scalar control input, ${\bf{1}}=(1,1),$ $u(\nabla u\cdot {\bf{1}})=u\sum_{i=1}^{2}u_{x_i}$ is the nonlinear convection term and 
$u_0$ is a given function. Application of \eqref{feq1.1} in the area of fluid mechanics is significant to study turbulence  behavior,  where $u$ is denoted as flow speed of fluid media, $\nu>0$ is the viscosity parameter which is analogous to the inverse of Reynolds number in the Navier Stokes system. The rate of flow or flux through the boundary is $v_2$ which is our control. When $\nu$ tends to zero in \eqref{feq1.1}, it models nonlinear wave propagation.

In literature, several  local stabilization results  for one dimensional Burgers' equation
are available, say for example, see, \cite{bk, bk1}
for distributed and Dirichlet boundary control, and \cite {bgs}
for Neumann boundary control under  sufficiently smallness assumption on the  initial
 data.  We refer to \cite {ik, iy} and \cite{lmt} for further references on local stabilization results including results on existence and uniqueness. Related to instantaneous control of 1D Burgers' equation, we refer to \cite{hv02}.\\
Local stabilization result for 2D viscous Burgers' equation is available in \cite{raymond2010} where a nonlinear feedback control law is applied which is obtained through solving Hamilton-Jacobi-Bellman (HJB) equation and using Riccati based optimal feedback control. 
For instance, authors  first formulate the two dimensional Burgers' equation in abstract form  as 
$$w_t=Aw+F(w)+Bv, \quad w(0)=w_0,$$ 
where $A$, with domain $D(A)$, is the  infinitesimal generator of an analytic semigroup on a Hilbert space $W$, $F(w)$ is the nonlinear term, $B$ is the operator from a control space $V$ into $W$. With corresponding
 cost functional  of the form 
$$J(w_v,v)<\infty \quad \text{where}\quad J(w_v,v)=\frac{1}{2}\int_{0}^{\infty}\Big(\norm{w_v}^2_{W}+\norm{v}^2_V\Big)\;dt,$$
the associated linear feedback control law  becomes  $v=-B^*Pw$, where $P$ is the
 solution to the algebraic Riccati equation for the Linear Quadratic Regulator (LQR) problem. Through solving HJB equation 
  using Taylor series expansion, one can obtain the 
 nonlinear feedback control law as $$v=-B^*Pw+B^*(A-BB^*P)^{-*}PF(w),$$ 
 where $(A-BB^*P)^{-*}$ is the inverse of $(A-BB^*P)^{*}$
 (see \cite{raymond2010} for more details).
 Later on, Buchot {\it {et al.}} \cite{raymond2015} have discussed local stabilization result in the case of partial information for the two dimensional Burgers' type equation.  
 Subsequently in \cite{Raymond06},  author has shown local stabilization results for the Navier-Stokes system around a nonconstant steady state solution by constructing a linear feedback control law for the corresponding linearized equation. This, in turn, locally stabilizes the original nonlinear system. 
 All the above mentioned  stabilization results are local in nature and are valid under smallness assumption on the data. 

Our attempt in this paper is to establish global stabilization result without smallness assumption on the data through the nonlinear Neumann control law using Lyapunov type functional. Such  global stabilization results for one dimensional Burgers' equation was earlier studied in  \cite{krstic1} and \cite {balogh} for both Dirichlet and Neumann boundary control laws.  When the coefficient of viscosity is unknown, an adaptive control for one dimensional Burgers' equation is discussed in 
 \cite{liu}, \cite{Smaoui}, and \cite{smaoui1}.  Although, effect of these control laws to their state are shown computationally using finite difference method and Chebychev collocation method, but convergence of numerical solution posses some serious difficulty because of the typical nonlinearity present in the system through nonlinear feedback laws.  Also in \cite{bk, bk1}, authors have considered finite element method to solve numerically local stabilization problem for 1D Burgers' equation without any convergence analysis. Subsequently in \cite{skakp1},  
optimal error estimates in the context of finite element method for the state variable
and superconvergence result for the feedback control laws are derived. For related analysis on Benjamin Bona Mahony Burgers' (BBM-Burgers') type equations, we  refer to \cite{skakp2}. Concerning extensive literature for one dimensional Burgers' and BBM-Burgers' problem, see the references in \cite{skakp1, skakp2}.

To the best of our knowledge, there is hardly any result on global stabilization for the two dimensional Burgers' equation. Further to continue our investigation keeping an eye on the Navier-Stokes system,  finite element method is applied  to 2D Burgers' equation, that is, the equation \eqref{feq1.1}.

 The major contributions of this article are summarized as follows:
 \begin{itemize}
 	\item  With the help of Lyapunov functional,  a nonlinear Neumann feedback control law  for the problem \eqref{feq1.1}-\eqref{feq1.3}  is derived and  global stabilization results 
 	 in $L^\infty(H^i)$ $(i=0,1,2)$ norms are established.
 	\item Based on $C^0$- conforming finite element method in spatial direction, optimal error estimates, (optimality with respect to approximation property)  for the state variable and  for the feedback control law are derived keeping time variable continuous.
\item Several  numerical examples including an example in which a part of boundary  is with Neumann control  and other part is with Dirichlet boundary condition are  given to illustrate our theoretical findings.
\end{itemize} 
For the rest of the article, denote $H^m(\Omega)=W^{m,2}(\Omega)$ to be the standard Sobolev space with norm $\norm{\cdot}_m,$ and seminorm $|\cdot|_m$. For $m=0,$ it corresponds to the usual $L^2$ norm and is denoted by $\norm{\cdot}$.
The space $L^p((0,T);X)$  $1\leq p\leq\infty,$ consists of all strongly measurable functions $v:[0,T] \rightarrow X $ with norm
$$\norm{v}_{L^p((0,T);X)}:=\left(\int_{0}^{T}\norm{v(t)}^p_X dt\right)^\frac{1}{p}<\infty \quad \text{for} \quad 1\leq p<\infty,$$ and 
$$\norm{v}_{L^\infty((0,T);X)}:=\operatorname*{ess\,sup}\limits_{0\leq t\leq T}\norm{v(t)}_X<\infty.$$
The rest of the paper is organized as follows. While Section $2$ is on problem formulation and preliminaries, Section $3$  focuses  on global stabilization results using a nonlinear feedback control law. Section $4$ deals with finite element approximation  for the semidiscrete system. Further, optimal error estimates are obtained for the state variable and convergence result is derived for the feedback control law. Finally, Section $5$ concludes with some numerical experiments.

\section{Preliminaries and problem formulation.}
This section focuses  on  some preliminary results to be used in our subsequent sections. Further, it  deals with  the Neumann control law using Lyapunov functional  and  with problem formulation for our global stabilizability and finite element analysis on our latter sections.

The following trace embedding result holds for 2D.\\
{\bf{Boundary Trace Imbedding Theorem} (page $164,$ \cite{adams2003}): }
There exists a bounded linear map
$$T:	H^1(\Omega)\hookrightarrow L^q(\partial \Omega) \quad \text{for}\quad 2\leq q<\infty $$ such that
\begin{equation}\label{1.30}
\norm{Ty}_{L^q(\partial \Omega)}\leq C\norm{y}_{H^1(\Omega)},
	\end{equation}
for each $y\in H^1(\Omega)$, with the constant $C$ depend on $q$ and $\Omega$. Also the following trace result holds\\
{\bf{Trace inequality} (\cite{lm68}): }
\begin{equation}\label{1.31}
\norm{Ty}_{H^s(\partial\Omega)}\leq C\norm{y}_{H^{s+\frac{1}{2}}(\Omega)}, \quad s\neq 1,\quad 0<s\leq \frac{3}{2}.
\end{equation}	
Below, we recall the following inequalities for our subsequent use:\\
{\bf{Friedrichs's inequality}:} For $y\in H^1(\Omega),$ there holds
\begin{equation}\label{feq1}
\norm{y}^2\leq C_F\Big(\norm{\nabla y}^2+\norm{y}^2_{L^2(\partial\Omega)}\Big),
\end{equation}
where $C_F>0$ is the  Friedrichs's  constant.\\
More precisely, in 2D we have
\begin{align*}
\int_{\Omega} y^2dx=\int_{\Omega}y^2\Delta\phi dx,
\end{align*}
where $\phi(x)=\frac{1}{4}|x|^2$ so that $\Delta\phi=1$.
Now integrate by parts to obtain
\begin{align*}
\int_{\Omega} y^2dx&=-2\int_{\Omega} y\nabla y\nabla\phi dx+\frac{1}{2}\int_{\partial\Omega}y^2(x\cdot n)\; d\Gamma\\
&=-\int_{\Omega} y\nabla y x dx+\frac{1}{2}\int_{\partial\Omega}y^2(x\cdot n)\; d\Gamma\\
&\leq \frac{1}{2}\int_{\Omega} y^2 dx+\frac{1}{2}\sup_{x\in\partial\Omega}|x|^2\int_{\Omega} |\nabla y|^2 dx+\frac{1}{2}\sup_{x\in\partial\Omega}|x|\int_{\partial\Omega}y^2\; d\Gamma.
\end{align*}
Therefore, it follows that 
\begin{align*}
\int_{\Omega} y^2dx\leq \sup_{x\in\partial\Omega}|x|^2\int_{\Omega} |\nabla y|^2 dx+\sup_{x\in\partial\Omega}|x|\int_{\partial\Omega}y^2\; d\Gamma.
\end{align*}
Hence, the Friedrichs's inequality constant can be taken as $C_F=\max\{\sup_{x\in\partial\Omega}|x|^2,\sup_{x\in\partial\Omega}|x|\}$.
{\bf{Gagliardo-Nirenberg inequality}} (see \cite{nirenberg59}):
For $w\in H^1(\Omega)$, we have
\begin{align*}
\norm{w}_{L^4}&\leq C\Big(\norm{w}^{1/2}\norm{\nabla w}^{1/2}+\norm{w}\Big), \quad \text{and for $w\in H^2(\Omega)$},\quad \text{we have}\\
\norm{\nabla w}_{L^4}&\leq C\Big(\norm{ w}^{1/4}\norm{\Delta w}^{3/4}+\norm{w}\Big).
\end{align*}
{\bf{Agmon's inequality}} (see \cite{agmons10}):
For $z\in H^2(\Omega),$ there holds
$$\norm{z}_{L^\infty}\leq C\Big(\norm{z}^\frac{1}{2}\norm{\Delta z}^\frac{1}{2}+\norm{z}\Big).$$	
Now the corresponding equilibrium or steady state problem of  \eqref{feq1.1}-\eqref{feq1.3} becomes: find $u^\infty$ as a solution of
\begin{align}
-\nu \Delta u^\infty+u^\infty(\nabla u^\infty\cdot {\bf{1}})&=0 \qquad\text{in} \quad \Omega \label{feq1.5},\\
\frac{\partial u^\infty}{\partial n}&=0 \quad \text{on} \quad \partial \Omega \label{feq1.6}.
\end{align}
Note that any constant $w_d$ satisfies \eqref{feq1.5}-\eqref{feq1.6}. Without loss of generality, we assume that $w_d\geq 0$. When $\nu$ is sufficiently small and initial condition $u_0$ is antisymmetric, the numerical solution of  \eqref{feq1.1}-\eqref{feq1.3} with $\frac{\partial u}{\partial n}=0$ may converge to a nonconstant steady state solution for which related references are given in \cite{skakp1}. We do not consider such  cases here. To achieve $$\lim_{t\to\infty} u(x,t)=w_d\quad \forall~ x\in\Omega,$$ it is enough to consider $\lim_{t\to\infty}w=0,$ where $w=u-w_d$ and $w$ satisfies
\begin{align}
&w_t-\nu \Delta w+w_d(\nabla w\cdot {\bf{1}})+w(\nabla w\cdot {\bf{1}})=0 \qquad\text{in}\quad (x,t)\in \Omega\times(0,\infty)\label{feq1.7},\\
&\frac{\partial w}{\partial n}(x,t)=v_2(x,t),\quad \text{on} \quad \partial \Omega\times(0,\infty)\label{eqn1.8},\\
&w(0)=u_0-w_d=w_0(\text{say})\quad\text{in}\quad\Omega \label{feq1.9}.
\end{align}
The motivation behind choosing the Neumann boundary control comes from the physical situation. Say for example,  in thermal problem, one cannot actuate the temperature $w$ on the boundary, but the heat flux $\frac{\partial w}{\partial n}$. This makes the stabilization problem nontrivial because $w_d$ is not asymptotically stable with zero Neumann boundary data. Concerning Dirichlet boundary control unlike in 1D \cite{krstic1}, it is not easy to get a concrete useful form of the control law. Although the control law $v_2$ derived in \eqref{feqx1}, is in invertible form so we can obtain $w$ in terms of $\frac{\partial w}{\partial n}$ on  the boundary  by solving cubic equation using Cardan's method, but that form is not useful for the stabilizability analysis.

For our analysis, the following compatibility conditions for $w_0$  on the boundary are required, namely; 
\begin{equation}\label{comp}
\frac{\partial w_0}{\partial n}=v_2(x,0) \quad \text{and} \quad \frac{\partial w_t}{\partial n}(x,0)=v_{2t}(x,0),
\end{equation}
 where $v_2(x,\cdot)$ is continuously differentiable at $t=0$ for almost all $x$. These conditions are required for the proof of  Lemmas \ref{flm4} and \ref{flm5}.
 
Now, the motivation for choosing  the control law comes from the construction of a Lyapunov functional of the following form $V(t)=\frac{1}{2}\int_{\Omega}w(x,t)^2\; dx$.
Hence,  on taking derivative with respect to time, we arrive at
\begin{align*}
\frac{dV}{dt}&=\int_{\Omega}w\Big(\nu \Delta w-w_d(\nabla w\cdot {\bf{1}})-w(\nabla w\cdot {\bf{1}})\Big)\;dx\\
&=-\nu\norm{\nabla w}^2+\nu\int_{\partial\Omega}\frac{\partial w}{\partial n}w\;d\Gamma-\int_{\Omega}w_d(\nabla w\cdot {\bf{1}})w\; dx-\int_{\Omega}w(\nabla w\cdot {\bf{1}})w\; dx.
\end{align*}
Using the Young's inequality,  it follows that
\begin{align} \label{estimate:linearised-term}
w_d\Big(\big(\nabla w\cdot {\bf{1}}\big),w\Big)=\frac{w_d}{2}\int_{\Omega}\Big((w^2)_{x_1}+(w^2)_{x_2}\Big)\;dx &=\frac{w_d}{2}\sum_{j=1}^{2}\int_{\partial\Omega}w^2\cdot\nu_j\; d\Gamma \notag\\
& \leq\frac{w_d}{\sqrt 2} \int_{\partial\Omega} w^2 \;d\Gamma
\leq  w_d\int_{\partial\Omega} w^2\; d\Gamma,
\end{align}
and
\begin{align} \label{estimate:nonlinear-term}
\int_{\Omega}w(\nabla w\cdot {\bf{1}})w\; dx\leq \frac{1}{3}\sum_{j=1}^{2}\int_{\partial\Omega}w^3\cdot\nu_j\; d\Gamma&\leq \frac{1}{3}\sqrt 2\int_{\partial\Omega}|w|^3\; d\Gamma\notag\\
&\leq {c_0}\int_{\partial\Omega}w^2\; d\Gamma+\frac{1}{18 c_0}\int_{\partial\Omega}w^4\; d\Gamma,
\end{align}
where $c_0$ is a positive constant.
Therefore, it follows that
\begin{align}
\frac{dV}{dt}\leq -\nu\norm{\nabla w}^2+\int_{\partial\Omega}\Big(\nu\frac{\partial w}{\partial n}+(w_d+{c_0}) w+\frac{1}{18  c_0}w^3\Big)w\;d\Gamma.
\end{align}
Now, choose the Neumann boundary feedback control law as
\begin{align}\label{feqx1}
v_2(x,t)=-\frac{1}{\nu}\Big(2(c_0+w_d)w+\frac{2}{9c_0}w^3\Big) \quad \text{on} \quad\partial\Omega,
\end{align}
to obtain
\begin{align*}
\frac{dV}{dt}&\leq -\nu\norm{\nabla w}^2-\big({c_0}+w_d\big)\int_{\partial \Omega}w^2\;d\Gamma
-\frac{1}{6 c_0}\int_{\partial \Omega}w^4\;d\Gamma\\
&\leq -\min\Big\{\nu,\big({c_0} +w_d\big)\Big\}\Big(\norm{\nabla w}^2+\norm{w}^2_{L^2(\partial\Omega)}\Big)\\
&\leq -\frac{2}{C_F}\min\Big\{\nu,\big({c_0}+w_d\big)\Big\}(\frac{1}{2}\norm{w}^2)
\leq -C_{Lyp}V,
\end{align*}
where $C_{Lyp}=\frac{2}{C_F}\min\Big\{\nu,\big({c_0}+w_d\big)\Big\}>0$.\\
Setting $B\big(v;w,\phi\big)$ as $B\big(v;w,\phi\big)=\Big(v\big(\nabla w\cdot {\bf{1}}\big),\phi\Big),$ 
$w$ satisfies a weak form of \eqref{feq1.7}-\eqref{feq1.9} as
\begin{align}
(w_t,v)+&\nu(\nabla w,\nabla v)+w_d\big(\nabla w\cdot {\bf{1}}, v\big)+B\big(w;w,v\big)\notag\\
&+\Big\langle(2 c_0+2w_d)w+\frac{2}{9c_0}w^3,v\Big\rangle_{\partial\Omega}\; =0\quad\forall~v\in H^1(\Omega)\label{feq1.10},
\end{align}
with $w(0)=w_0,$ where $\langle v,w\rangle_{\partial\Omega}:=\int_{\partial \Omega}vw \;d\Gamma$.\\
For our subsequent analysis, we assume that there exists a unique weak solution $w$ of \eqref{feq1.10} satisfying the following regularity results
\begin{align}\label{ex}
\norm{w(t)}^2_2+\norm{w_t(t)}^2_1+\int_{0}^{t}\norm{ w_t(s)}^2_2 ds\leq C.
\end{align}
For existence and uniqueness  with continuous dependence property of  one dimensional Burgers' equation with similar type nonlinearity, see, \cite{liu},\cite{iy} and their arguments can be modified to prove the wellposedness of the problem \eqref{feq1.10}. For regularity results, the energy method applied in our section 3 can be appropriately modified to prove 
\eqref{ex}. Therefore, we shall not pursue it further in this article.

Throughout the paper $C$ is a generic positive constant. 

\section{Stabilization results}
In this section, we establish global stabilization results of the continuous problem \eqref{feq1.1}-\eqref{feq1.3}. More precisely, exponential stabilization results  for the state variable $w(t)$ are shown for the modified problem \eqref{feq1.7}-\eqref{feq1.9}, where feedback control $v_2$ is given in \eqref{feqx1}. Moreover, additional regularity results are established assuming compatibility conditions, which are crucial for proving optimal error estimates for the state variable. 

Our results of this section are based on energy arguments using exponential weight functions. For similar analysis, see, \cite{temam},  \cite{HR82}, \cite{Sobolevskii}, \cite{He}, and \cite{goswami}.

Throughout this section, all the results hold with the same decay rate $\alpha$:
\begin{equation}\label{decay}
0\leq\alpha\leq\frac{1}{C_F}\min\Big\{{\nu},(c_0+w_d)\Big\}.
\end{equation}
\begin{lemma}\label{flm1}
	Let $w_0\in L^2(\Omega)$. Then, there holds
	\begin{align*}
	\norm{w(t)}^2+\beta e^{-2\alpha t}\int_{0}^{t}e^{2\alpha s}\Big(\norm{\nabla w(s)}^2+\norm{w(s)}^2_{L^2(\partial\Omega)}+\frac{1}{ 3\beta c_0}\norm{w(s)}^4_{L^4(\partial\Omega)} \Big)ds\leq e^{-2\alpha t}\norm{w_0}^2,
	\end{align*}
	where $\beta= 2\min\{(\nu-\alpha C_F), (c_0+w_d-\alpha C_F)\}>0$, and $C_F>0$ is the constant in the Friedrichs's inequality \eqref{feq1}.
\end{lemma}
\begin{proof}
	Set $v=e^{2\alpha t}w$ in \eqref{feq1.10} to obtain
	\begin{align}
	\frac{d}{dt}\norm{e^{\alpha t}w}^2-&2\alpha \norm{e^{\alpha t}w}^2+2\nu \norm{e^{\alpha t}\nabla w}^2+2e^{2\alpha t}\int_{\partial \Omega}\Big(( 2 c_0+2w_d)w^2+\frac{2}{9c_0}w^4\Big)\; d\Gamma\notag\\
	&=-2w_de^{2\alpha t}\Big(\big(\nabla w\cdot {\bf{1}}\big),w\Big)-2e^{2\alpha t}B\big(w;w,w\big)
	\label{feq1.11}.
	\end{align}
	For the first term on the right hand side of \eqref{feq1.11}, we use \eqref{estimate:linearised-term} to  bound it as
	\begin{align}
	2w_de^{2\alpha t}\Big(\big(\nabla w\cdot {\bf{1}}\big),w\Big) 
	\leq \sqrt 2 w_de^{2\alpha t}\int_{\partial\Omega} w^2 d\Gamma 
	\leq 2 w_de^{2\alpha t}\int_{\partial\Omega} w^2 d\Gamma\label{feq1.12}.
	\end{align}
	For the second term on the right hand side of \eqref{feq1.11},  a use of   \eqref{estimate:nonlinear-term}  with the Young's inequality yields
	\begin{align}
	2e^{2\alpha t}B\big(w;w,w\big) 
	\leq \frac{2}{3}e^{2\alpha t}\sqrt 2\int_{\partial\Omega}|w|^3\; d\Gamma 
	\leq 2 c_0 e^{2\alpha t}\int_{\partial\Omega}w^2\; d\Gamma+\frac{1}{9c_0}e^{2\alpha t}\int_{\partial\Omega}w^4 \;d\Gamma\label{feq1.13}.
	\end{align}
	Now, using the Friedrichs's inequality \eqref{feq1}, it follows that
	\begin{equation}\label{feq1.14}
	-2\alpha e^{2\alpha t}\norm{w}^2 \geq -2\alpha e^{2\alpha t}C_F\Big(\norm{\nabla w}^2+\norm{w}^2_{L^2(\partial\Omega)}\Big).
	\end{equation}
	Hence, from \eqref{feq1.11}, we arrive using \eqref{feq1.12}- \eqref{feq1.14} at
	\begin{align}\label{feq1.15}
	\frac{d}{dt}\norm{e^{\alpha t}w}^2+2(\nu-\alpha C_F)\norm{e^{\alpha t}\nabla w}^2+&2 e^{2\alpha t}\Big(\big(c_0+w_d- \alpha C_F\big)\int_{\partial\Omega}w^2\; d\Gamma\notag\\
	&+\frac{1}{6 c_0}\int_{\partial\Omega}w^4\; d\Gamma\Big) 
	\leq 0.
	\end{align}
Since decay rate satisfy \eqref{decay}, the coefficients on the left hand side of \eqref{feq1.15} are non-negative.
	Integrate \eqref{feq1.15} with respect to time from $0$ to $t,$ and then, multiply the resulting inequality by $e^{-2\alpha t}$ to obtain
	\begin{align*}
	\norm{w(t)}^2+2(\nu-\alpha C_F)e^{-2\alpha t}\int_{0}^{t}e^{2\alpha s}\norm{\nabla w(s)}^2 ds&+2 e^{-2\alpha t}\int_{o}^{t}e^{2\alpha s}\Big(\big(c_0+w_d- \alpha C_F\big)\norm{w(s)}^2_{L^2(\partial\Omega)}\\
	&\qquad+\frac{1}{6 c_0}\norm{w(s)}^4_{L^4(\partial\Omega)}\Big) \;ds\leq e^{-2\alpha t}\norm{w_0}^2.
	\end{align*}
	This completes the proof.
\end{proof}
\begin{remark}\label{frm2.1}
	The above Lemma also holds for $\alpha=0,$ that is,
	\begin{equation}\label{feqn1.2}
	\norm{w(t)}^2+2\nu\int_{0}^{t}\norm{\nabla w(s)}^2 ds+2 \int_{0}^{t}\Bigg(\int_{\partial\Omega}\Big(\big(c_0+ w_d\big)w(s)^2+\frac{1}{6c_0}w(s)^4\Big) \;d\Gamma\Bigg) ds\leq \norm{w_0}^2.
	\end{equation}
	Moreover,
	by the Friedrichs's inequality, it follows that
	\begin{equation*}
	e^{-2\alpha t}\int_{0}^{t}e^{2\alpha s}\norm{w(s)}^2 ds\leq Ce^{-2\alpha t}\norm{w_0}^2.
	\end{equation*}
\end{remark}
\begin{remark}\label{rm}
Now instead of taking the control on the whole boundary, if we take the above mentioned Neumann control on some part of the boundary ($\Gamma_N$) where $\Gamma_N$ has nonzero measure with remaining part zero Dirichlet boundary condition, still the stabilization result holds.
For instance, consider $\partial\Omega=\Gamma_D\cup \Gamma_N$ with $\Gamma_D\cap \Gamma_N=\phi$, where $\Gamma_D$ and $\Gamma_N$ are sufficiently smooth. 
With this setting, from \eqref{feq1.11}, we arrive at
\begin{align}
\frac{d}{dt}\norm{e^{\alpha t}w}^2-&2\alpha \norm{e^{\alpha t}w}^2+2\nu \norm{e^{\alpha t}\nabla w}^2+2e^{2\alpha t}\int_{\Gamma_N}\Big((2 c_0+2w_d)w^2+\frac{2}{9c_0}w^4\Big)\; d\Gamma\notag\\
&=-2w_de^{2\alpha t}\Big(\big(\nabla w\cdot {\bf{1}}\big),w\Big)-2e^{2\alpha t}B\big(w;w,w\big)\notag\\
&\leq 2 w_de^{2\alpha t}\int_{\Gamma_N} w^2 d\Gamma+2 c_0 e^{2\alpha t}\int_{\Gamma_N}w^2 d\Gamma+\frac{1}{9c_0}e^{2\alpha t}\int_{\Gamma_N}w^4 d\Gamma
\label{xfeq1.11}.
\end{align} 
Using Friedrichs's inequality $\norm{v}^2\leq C_F\Big(\norm{\nabla v}^2+\norm{v}^2_{L^2(\Gamma_N)}\Big)$, we obtain
\begin{align*}
\frac{d}{dt}\norm{e^{\alpha t}w}^2+2(\nu-\alpha C_F)\norm{e^{\alpha t}\nabla w}^2+&2 e^{2\alpha t}\Big(\big(c_0+w_d-2\alpha C_F\big)\int_{\Gamma_N}w^2 d\Gamma \notag\\
&\qquad +\frac{1}{6 c_0}\int_{\Gamma_N}w^4 d\Gamma\Big) \leq 0.
\end{align*}
Proceed as before to complete the rest of the proof for $L^2$- stabilization result.
In higher order norm, stabilization result also holds similarly when control works on some part of the boundary.
\end{remark}
\begin{lemma}\label{flm2}
	Let $w_0\in H^1(\Omega).$ Then, for $C=C(\norm{w_0}_1)$ there holds 
	\begin{align*}
	\Big(\norm{\nabla w(t)}^2&+\frac{2(c_0+w_d)}{\nu}\norm{w(t)}^2_{L^2(\partial\Omega)}+\frac{1}{9\nu c_0}\norm{w(t)}^4_{L^4(\partial\Omega)}\Big)+\nu e^{-2\alpha t}\int_{0}^{t}\norm{e^{\alpha s}\Delta w(s)}^2\; ds\\&\leq Ce^Ce^{-2\alpha t}.
	\end{align*}	
\end{lemma}
\begin{proof}
	Form an $L^2$-inner product between \eqref{feq1.7} and $-e^{2\alpha t}\Delta w$ to obtain
	\begin{align}
	\frac{d}{dt}\norm{e^{\alpha t}\nabla w}^2-&2\alpha e^{2\alpha t}\norm{\nabla w}^2+2\nu\norm{e^{\alpha t}\Delta w}^2+\frac{2}{\nu}\int_{\partial\Omega}e^{2\alpha t}\Big(2(c_0+w_d)w+\frac{2}{9c_0}w^3\Big)w_t\; d\Gamma\notag\\
	&=2e^{2\alpha t}w_d(\nabla w\cdot {\bf{1}},\Delta w)+2e^{2\alpha t}B(w;w,\Delta w)\label{feq1.21}.
	\end{align}
	The fourth term on the left hand side of \eqref{feq1.21} can be rewritten as
	\begin{align*}
	\frac{2}{\nu}\int_{\partial\Omega}&e^{2\alpha t}\Big(2(c_0+w_d)w+\frac{2}{9c_0}w^3\Big)w_t\; d\Gamma\\
	&= \frac{d}{dt}\Big(\frac{2(c_0+w_d)}{\nu}\norm{e^{\alpha t}w}^2_{L^2(\partial\Omega)}+\frac{1}{9\nu c_0}\big(e^{2\alpha t}\norm{w}^4_{L^4(\partial\Omega)}\big)\Big)\notag\\
	&\qquad-2\alpha e^{2\alpha t}\Big(\frac{2(c_0+w_d)}{\nu}\norm{w}^2_{L^2(\partial\Omega)}+\frac{1}{9\nu c_0}\norm{w}^4_{L^4(\partial\Omega)}\Big).
	\end{align*}
	The terms on the right hand side of \eqref{feq1.21} are bounded by
	\begin{align*}
	2e^{2\alpha t}w_d(\nabla w\cdot{\bf{1}},\Delta w)\leq \frac{\nu}{2}\norm{e^{\alpha t}\Delta w}^2+\frac{2}{\nu}e^{2\alpha t}w_d^2\norm{\nabla w}^2,
	\end{align*}
	and using Gagliardo-Nirenberg inequality and Lemma \ref{flm1}, by
	\begin{align*}
	2e^{2\alpha t}B(w;w,\Delta w)&\leq Ce^{2\alpha t}\norm{w}_{L^4}\norm{\nabla w}_{L^4}\norm{\Delta w}\\
	&\leq Ce^{2\alpha t}\Big(\norm{w}^\frac{1}{2}\norm{\nabla w}^\frac{1}{2}+\norm{w}\Big)\Big(\norm{w}^\frac{1}{4}\norm{\Delta w}^\frac{3}{4}+\norm{w}\Big)\norm{\Delta w}\\
	&
	\leq Ce^{2\alpha t}\Big(\norm{w}^\frac{3}{4}\norm{\nabla w}^\frac{1}{2}\norm{\Delta w}^\frac{7}{4}+\norm{w}^\frac{5}{4}\norm{\Delta w}^\frac{7}{4}+\norm{w}^\frac{3}{2}\norm{\nabla w }^\frac{1}{2}\norm{\Delta w}\\
	&\qquad+\norm{w}^2\norm{\Delta w}\Big)\\
	&\leq \frac{\nu}{2}\norm{e^{\alpha t}\Delta w}^2+Ce^{2\alpha t}\norm{w}^2\norm{\nabla w}^4+Ce^{2\alpha t}\norm{w}^2+Ce^{2\alpha t}\norm{w}^2\norm{\nabla w}^2.
	\end{align*}
	Finally, from \eqref{feq1.21}, we arrive at
	\begin{align}
	\frac{d}{dt}\Big(e^{2\alpha t}\big(\norm{\nabla w}^2&+\frac{2(c_0+w_d)}{\nu}\norm{w}^2_{L^2(\partial\Omega)}+\frac{1}{9\nu c_0}\norm{w}^4_{L^4(\partial\Omega)}\big)\Big)+\nu\norm{e^{\alpha t}\Delta w}^2\notag\\
	&\leq 2\alpha e^{2\alpha t}\Big(\frac{2(c_0+w_d)}{\nu}\norm{w}^2_{L^2(\partial\Omega)}+
	\frac{1}{9\nu c_0}\norm{w}^4_{L^4(\partial\Omega)}\Big)+\frac{2}{\nu}e^{2\alpha t}w_d^2\norm{\nabla w}^2\notag\\
	&\qquad +Ce^{2\alpha t}\norm{w}^2+Ce^{2\alpha t}\norm{w}^2\norm{\nabla w}^2
	+Ce^{2\alpha t}\norm{w}^2\norm{\nabla w}^4\label{feq1.22}.
	\end{align}
	Integrate the above inequality from $0$ to $t,$ and then use the Gr\"onwall's inequality with Lemma \ref{flm1} to obtain
	\begin{align*}
	e^{2\alpha t}\big(\norm{\nabla w(t)}^2&+\frac{2(c_0+w_d)}{\nu}\norm{w(t)}^2_{L^2(\partial\Omega)}+\frac{1}{9\nu c_0}\norm{w(t)}^4_{L^4(\partial\Omega)}\big)+\nu\int_{0}^{t}\norm{e^{\alpha s}\Delta w(s)}^2\; ds\\&\leq C\Big(\norm{w_0}^2_1+\norm{w_0}^2_{L^2(\partial\Omega)}+\norm{w_0}^4_{L^4(\partial\Omega)}\Big)\exp\Big(C\int_{0}^{t}\norm{w}^2\big(1+\norm{\nabla w}^2\big) ds\Big).
	\end{align*}
	Use Remark \ref{frm2.1} for the integral term under the exponential sign, and then multiply the resulting inequality by $e^{-2\alpha t}$ to complete the rest of the proof.
\end{proof}
\begin{lemma}\label{flm3}
	Let $w_0\in H^1(\Omega)$. Then, there exists a positive constant  $C=C\Big(\norm{w_0}_1\Big)$ such that the following estimate holds.
	\begin{align*}
	\Big(\nu \norm{\nabla w(t)}^2&+2(c_0+w_d)\norm{w(t)}^2_{L^2(\partial\Omega)}+\frac{1}{9c_0}\norm{w(t)}^4_{L^4(\partial\Omega)}\Big)+e^{-2\alpha t}\int_{0}^{t}e^{2\alpha s}\norm{w_t(s)}^2 ds\leq Ce^Ce^{-2\alpha t}.
	\end{align*}
\end{lemma}
\begin{proof}
	Choose $v=e^{2\alpha t}w_t$ in \eqref{feq1.10} to obtain
	\begin{align}
	2\norm{e^{\alpha t}w_t}^2+\nu \frac{d}{dt}\norm{e^{\alpha t}\nabla w}^2&-2\nu\alpha\norm{e^{\alpha t}\nabla w}^2+2\int_{\partial\Omega}\Big(2(c_0+w_d)w+\frac{2}{9c_0}w^3\Big)e^{2\alpha t}w_t \;d\Gamma\notag\\
	&=-2w_de^{2\alpha t}\big(\nabla w\cdot{\bf{1}},w_t\big)-2e^{2\alpha t}B\big(w;w,w_t\big)\label{feq1.31}.
	\end{align}
	The terms on the right hand side of \eqref{feq1.31} are bounded by
	\begin{align*}
	2w_de^{2\alpha t}\big(\nabla w\cdot{\bf{1}},w_t\big)\leq \frac{1}{2}e^{2\alpha t}\norm{w_t}^2+4e^{2\alpha t}w_d^2\norm{\nabla w}^2,
	\end{align*}
	and using Gagliardo-Nirenberg inequality and Lemma \ref{flm1}, by
	\begin{align*}
	2e^{2\alpha t}B\big(w;w,w_t\big)&\leq Ce^{2\alpha t}\norm{w}_{L^4}\norm{\nabla w}_{L^4}\norm{w_t}\\
	&\leq Ce^{2\alpha t}\Big(\norm{w}^\frac{1}{2}\norm{\nabla w}^\frac{1}{2}+\norm{w}\Big)\Big(\norm{w}^\frac{1}{4}\norm{\Delta w}^\frac{3}{4}+\norm{w}\Big)\norm{w_t}\\
	&\leq Ce^{2\alpha t}\Big(\norm{w}^\frac{3}{4}\norm{\nabla w}^\frac{1}{2}\norm{\Delta w}^\frac{3}{4}\norm{w_t}+\norm{w}^\frac{5}{4}\norm{\Delta w}^\frac{3}{4}\norm{w_t}\\
	&\qquad+\norm{w}^\frac{3}{2}\norm{\nabla w}^\frac{1}{2}\norm{w_t}+\norm{w}^2\norm{w_t}\Big)\\
	&\leq \frac{1}{2}e^{2\alpha t}\norm{w_t}^2+Ce^{2\alpha t}\norm{w}^2\norm{\nabla w}^4+Ce^{2\alpha t}\norm{\Delta w}^2+Ce^{2\alpha t}\norm{w}^2\\
	&\qquad+Ce^{2\alpha t}\norm{w}^2\norm{\nabla w}^2.
	\end{align*}
	Hence, rewriting the boundary integral term in \eqref{feq1.31} as in previous Lemma \ref{flm2}, we arrive from \eqref{feq1.31} at
	\begin{align*}
	\frac{d}{dt}\Big(e^{2\alpha t}\big(\nu \norm{\nabla w}^2&+2(c_0+w_d)\norm{w}^2_{L^2(\partial\Omega)}+\frac{1}{9c_0}\norm{w}^4_{L^4(\partial\Omega)}\big)\Big)+\norm{e^{\alpha t}w_t}^2\\
	&\leq Ce^{2\alpha t}\Big(\norm{w}^2_{L^2(\partial\Omega)}+\norm{w}^4_{L^4(\partial\Omega)}+\norm{\nabla w}^2+\norm{\Delta w}^2+\norm{w}^2\norm{\nabla w}^2\\
	&\hspace{2cm}+\norm{w}^2\norm{\nabla w}^4+\norm{w}^2\Big).
	\end{align*}
	Apply Lemmas \ref{flm1} and \ref{flm2}, and the Gr\"onwall's inequality to the above inequality to complete the rest of the proof.
\end{proof}
\begin{lemma}\label{flm4}
	Let $w_0\in H^2(\Omega)$. Then there exists a positive constant  $C=C\Big(\norm{w_0}_2\Big)$ such that
	\begin{align*}
	\norm{w_t(t)}^2+\norm{\Delta w(t)}^2+\nu e^{-2\alpha t}\int_{0}^{t}e^{2\alpha s}\norm{\nabla w_t(s)}^2 ds
	&+2e^{-2\alpha t}\int_{0}^{t}e^{2\alpha s}\Big(2(c_0+w_d)\norm{w_t(s)}^2_{L^2(\partial\Omega)}\\
	&+\frac{2}{3c_0}\norm{w(s)w_t(s)}^2_{L^2(\partial\Omega)}\Big) ds\leq Ce^Ce^{-2\alpha t}.
	\end{align*}
\end{lemma}
\begin{proof}
	Differentiate \eqref{feq1.7} with respect to $t$ and then take the inner product with $e^{2\alpha t}w_t$ to obtain 
	\begin{align}
	\frac{d}{dt}\big(&\norm{e^{\alpha t}w_t}^2\big)-2\alpha \norm{e^{\alpha t}w_t}^2+2\nu\norm{e^{\alpha t}\nabla w_t}^2+2\int_{\partial\Omega}\big(2(c_0+w_d)w_t^2+\frac{2}{3c_0}w^2w_t^2\big)e^{2\alpha t}\;d\Gamma\notag\\
	&= -2e^{2\alpha t}\Big(B\big(w_t;w,w_t\big)+B\big(w;w_t,w_t\big)\Big)-2w_de^{2\alpha t}\big(\nabla w_t\cdot{\bf{1}},w_t\big)\label{feq2.1}.
	\end{align}
	The right hand side terms in \eqref{feq2.1} are bounded by
	\begin{align*}
	-2e^{2\alpha t}&\Big(\big(\nabla w_t\cdot{\bf{1}},w_t\big)+B\big(w_t;w,w_t\big)+B\big(w;w_t,w_t\big)\Big)\\
	&\leq 2w_de^{2\alpha t}\norm{\nabla w_t}\norm{w_t}+Ce^{2\alpha t}\norm{w_t}_{L^4}\norm{\nabla w}\norm{w_t}_{L^4}+Ce^{2\alpha t}\norm{w}_{L^4}\norm{\nabla w_t}\norm{w_t}_{L^4}\\
	&\leq Ce^{2\alpha t}\norm{\nabla w_t}\norm{w_t}+Ce^{2\alpha t}\Big(\norm{w_t}^\frac{1}{2}\norm{\nabla w_t}^\frac{1}{2}+\norm{w_t}\Big)^2\norm{\nabla w}\\
	&\qquad+C\Big(\norm{w}^\frac{1}{2}\norm{\nabla w}^\frac{1}{2}+\norm{w}\Big)\Big(\norm{w_t}^\frac{1}{2}\norm{\nabla w_t}^\frac{1}{2}+\norm{w_t}\Big)\norm{\nabla w_t}\\
	&\leq \nu\norm{e^{\alpha t}\nabla w_t}^2+Ce^{2\alpha t}\Big(\norm{w_t}^2+\norm{w_t}^2\norm{\nabla w}^2+\norm{w}^2\norm{\nabla w}^2\norm{w_t}^2+\norm{w}^2\norm{\nabla w}^2\\
	&\hspace{4cm}+\norm{w_t}^2\norm{w}^2+\norm{w_t}^2\norm{w}^4+\norm{w}^2\Big).
	\end{align*}
	Hence, from \eqref{feq2.1}, we arrive at
	\begin{align}
	\frac{d}{dt}(\norm{e^{\alpha t}w_t}^2)+&\nu \norm{e^{\alpha t}\nabla w_t}^2+2e^{2\alpha t}
	\Big(2(c_0+w_d)\norm{w_t}^2_{L^2(\partial\Omega)}+\frac{2}{3c_0}\norm{ww_t}^2_{L^2(\partial\Omega)}\Big)\notag\\
	&\leq Ce^{2\alpha t}\Big(\norm{w_t}^2+\norm{w_t}^2\norm{\nabla w}^2+\norm{w}^2\norm{\nabla w}^2\norm{w_t}^2+\norm{w}^2\norm{\nabla w}^2\notag\\
	&\hspace{4cm}+\norm{w_t}^2\norm{w}^2+\norm{w_t}^2\norm{w}^4+\norm{w}^2\Big)\label{feq2.2}.
	\end{align}
	To calculate $\norm{w_t(0)},$
	take the inner product between \eqref{feq1.7} and $w_t$; and use compatibility condition to obtain
	\begin{align*}
	\norm{w_t(0)}^2\leq C\Big(\norm{\nabla w_0}^2+\norm{\Delta w_0}^2+\norm{w_0}^2\norm{\nabla w_0}^4\Big).
	\end{align*}
	Integrate the inequality \eqref{feq2.2} from $0$ to $t$ and then use Lemmas \ref{flm1}-\ref{flm3} to complete  the proof of $\norm{w_t(t)}.$  From this the estimate of $\norm{\Delta w(t)}$ follows. Altogether it completes   the rest of the proof.
\end{proof}
\begin{lemma}\label{flm5}
	Let $w_0\in H^3(\Omega).$ Then there exists a positive constant $C=C\Big(\norm{w_0}_3\Big)$ such that
	\begin{align*}
	\norm{\nabla w_t(t)}^2+\big(2(c_0+w_d)\norm{w_t(t)}^2_{L^2(\partial\Omega)}&+\frac{2}{3c_0}\norm{w(t)w_t(t)}^2_{L^2(\partial\Omega)}\big)+\nu e^{-2\alpha t}\int_{0}^{t}e^{2\alpha s}\norm{\Delta w_t(s)}^2 ds\\
	&\leq Ce^{C(\norm{w_0}_2)}e^{-2\alpha t}.
	\end{align*}
\end{lemma}
\begin{proof}
	Differentiate \eqref{feq1.7} with respect to $t$ and then take inner product with $-e^{2\alpha t}\Delta w_t$ to obtain
	\begin{align}
	\frac{d}{dt}\norm{e^{\alpha t}\nabla w_t}^2&-2\alpha\norm{e^{\alpha t}\nabla w_t}^2+2\nu\norm{e^{\alpha t}\Delta w_t}^2+\frac{d}{dt}\int_{\partial\Omega}e^{2\alpha t}\Big(2(c_0+w_d)w_t^2+\frac{2}{3c_0}w^2w_t^2\Big) d\Gamma\notag\\
	&\leq 2e^{2\alpha t}w_d(\nabla w_t\cdot{\bf{1}},\Delta w_t)+2e^{2\alpha t}B\big(w_t,w,\Delta w_t\big)+2e^{2\alpha t}B\big(w;w_t,\Delta w_t\big)\notag\\
	&\qquad+C\int_{\partial\Omega}e^{2\alpha t}\Big(w_t^2+ww_t^3+w^2w_t^2\Big)\;d\Gamma\label{feq2.5}.
	\end{align}
	The first three terms on the right hand side of \eqref{feq2.5} are bounded by
	\begin{align*}
	2e^{2\alpha t}w_d(\nabla w_t\cdot{\bf{1}},\Delta w_t)\leq \frac{\nu}{3}\norm{e^{\alpha t}\Delta w_t}^2+Ce^{2\alpha t}\norm{\nabla w_t}^2,
	\end{align*}
	and using Gagliardo-Nirenberg inequality and Lemma \ref{flm1}, by
	\begin{align*}
	2e^{2\alpha t}&\Big(B\big(w_t;w,\Delta w_t\big)+
	B\big(w;w_t,\Delta w_t\big)\Big)\\
	&\leq 
	Ce^{2\alpha t}\Big(\norm{w_t}_{L^4}\norm{\nabla w}_{L^4}\norm{\Delta w_t}+\norm{w}_{L^4}\norm{w_t}_{L^4}\norm{\Delta w_t}\Big)\\
	&\leq \frac{2\nu}{3}\norm{e^{\alpha t}\Delta w_t}^2+Ce^{2\alpha t}\norm{w_t}^2\Big(\norm{w}^2_2+\norm{w}^2\norm{\nabla w}^4\Big)+Ce^{2\alpha t}\norm{\nabla w_t}^2\Big(\norm{w}^2+\norm{\Delta w}^2\Big).
	\end{align*}
	The boundary terms on the right hand side of \eqref{feq2.5} are bounded by
	\begin{align*}
	C\int_{\partial\Omega}e^{2\alpha t}\Big(w_t^2+ww_t^3+w^2w_t^2\Big)\;d\Gamma&\leq C\int_{\partial\Omega}e^{2\alpha t}\Big(w_t^2+w^2w_t^2\Big)\;d\Gamma+Ce^{2\alpha t}\norm{w_t}^4_{L^4(\partial\Omega)}.
	\end{align*}
	Therefore, from \eqref{feq2.5}, we arrive at
	\begin{align*}
	\frac{d}{dt}\Big(\norm{e^{\alpha t}\nabla w_t}^2&+2(c_0+w_d)\norm{e^{\alpha t}w_t}^2_{L^2(\partial\Omega)}+\frac{2}{3c_0}\int_{\partial\Omega}e^{2\alpha t}w^2w_t^2\;d\Gamma\Big)+\nu\norm{e^{\alpha t}\Delta w_t}^2\\
	&\leq Ce^{2\alpha t}\norm{w_t}^2\Big(\norm{w}^2_2+\norm{w}^2\norm{\nabla w}^4\Big)+Ce^{2\alpha t}\norm{\nabla w_t}^2\Big(1+\norm{w}^2+\norm{\Delta w}^2\Big)\\
	&\qquad+\int_{\partial\Omega}e^{2\alpha t}\Big(w_t^2+w^2w_t^2\Big)\;ds+Ce^{2\alpha t}\Big(\norm{w_t}^4+\norm{\nabla w_t}^4\Big).
	\end{align*}
	Integrate the above inequality from $0$ to $t$ and then apply the Gr\"onwall's inequality along with Lemmas \ref{flm1}-\ref{flm4} to obtain
	\begin{align}
	\Big(\norm{e^{\alpha t}\nabla w_t(t)}^2&+2(c_0+w_d)\norm{e^{\alpha t}w_t(t)}^2_{L^2(\partial\Omega)}+\frac{2}{3c_0}\int_{\partial\Omega}e^{2\alpha t}w(t)^2w_t(t)^2\;d\Gamma\Big)+\nu \int_{0}^{t}\norm{e^{\alpha s}\Delta w_t(s)}^2 ds\notag\\
	&\leq C\Big(\norm{\nabla w_t(0)}^2+\norm{w_t(0)}_{L^2(\partial\Omega)}+\norm{w(0)w_t(0)}^2_{L^2(\partial\Omega)}\Big)\notag\\
	&\qquad\exp\Big(C\int_{0}^{t}\big(\norm{w(s)}^2+\norm{\Delta w(s)}^2+\norm{\nabla w_t(s)}^2\big)\;ds\Big)\label{feqx2.5}.
	\end{align}
	Differentiate \eqref{feq1.7} with respect to $x_1$ and $x_2$ and applying compatibility condition to arrive at
	$\norm{\nabla w_t(0)}\leq C\norm{w_0}_3$. Also, by \eqref{1.30}, $\norm{w_t(0)}_{L^2(\partial\Omega)}\leq C\norm{w_t(0)}^2_1$ and 
	$\norm{w(0)w_t(0)}^2_{L^2(\partial\Omega)}\leq C\norm{w(0)}_{L^4(\partial\Omega)}\norm{w_t(0)}^2_1.$\\
	Again, use of Lemmas \ref{flm1}, \ref{flm2} and \ref{flm4} for the above inequality \eqref{feqx2.5} completes the proof.
\end{proof}
\section{Finite element method}
In this section, we discuss semidiscrete Galerkin approximation keeping  the  time variable continuous and show global stabilization result for the spatially discrete or semidiscrete solution. Further optimal error estimates ( optimality with respect to approximation property)  for both state variable and feedback controller are derived. With respect to $H^s$ regularity of $w$, the optimal error estimate satisfy $\norm{w(t)-w_h(t)}_{H^m(\Omega)}\leq Ch^{\min(s,k+1)-m}\norm{w(t)}_{H^s(\Omega)}$, where $m$, $s$ ( $0\leq m<s$)  are integers and $k$ is the degree of polynomial used, Here $m=0$ or $1$, $k=1$ and $s=2$.\\
For our semidiscrete analysis we assume $\Omega$ is a convex polygonal domain with boundary
$\partial\Omega$, otherwise we can always get a polygonal domain $\Omega_h\subset \Omega$, where union of quasi-uniform  triangulation determines $\Omega_h$ with boundary vertices on $\partial\Omega$. For more details see \cite{thomee}.
Given a regular triangulation  $\mathcal{T}_h$ of $\overline{\Omega}$, 
let $h_K=\text{diam}(K)$ for all $K\in \mathcal{T}_h$ and $h=\displaystyle\max_{ K\in \mathcal{T}_h} h_K$.\\
Set
$$V_h=\left\{v_h\in C^0(\overline{\Omega} ):\hspace{0.1cm} v_h\Big|_K \in \mathcal{P}_{1}(K) \quad \forall\hspace{0.1cm} K\in \mathcal{T}_{h}\right\}.$$
We shall assume further that  the following inverse property hold for each $v_h\in V_h$ and $p\in [2,\infty],$ see, \cite{ciarlet}
\begin{equation}\label{inverse-property}
\norm{v_h}_{L^p(\Omega)} \leq C \;h^{2(\frac{1}{p}-\frac{1}{2})} \;\norm{v_h}.
\end{equation}
The semidiscrete approximation corresponding to the problem  \eqref{feq1.10} is to seek $w_h(t)=w_h(\cdot,t)\in V_h$ such that
\begin{align}
(w_{ht},\chi)+&\nu(\nabla w_h,\nabla \chi)+w_d\big(\nabla w_h\cdot {\bf{1}},\chi\big)+B\big(w_h;w_h,\chi\big)\notag\\
&\qquad+\int_{\partial \Omega}\Big(2(c_0+w_d)w_h+\frac{2}{9c_0}w_h^3\Big)\chi\; d\Gamma
=0,\quad \forall~ \chi\in V_h\label{feqn3.1}
\end{align}
with $w_h(0)=P_h u_0-w_d=w_{0h}$ (say), an approximation of $w_0,$
where, $P_hu_0$ is the $H^1$ projection of $u_0$ onto $V_h$ such that
\begin{equation}
\norm{u_0-u_{0h}}_j\leq Ch^{2-j}\norm{u_0}_2\quad j=0,1.
\end{equation}
Since $V_h$ is finite dimensional, \eqref{feqn3.1} leads to a system of nonlinear
ODEs. Hence, an application of Picard's theorem ensures the existence of a unique
solution locally, that is, there exists an interval $(0,t_h)$   such that $w_h$
exists for $t\in (0,t_h)$. Then, using the boundedness of the discrete solution
from Lemma \ref{flm3.1} below, the continuation arguments
yields existence of a unique solution for all $t>0$.

In a similar fashion as in continuous case, the following stabilization result holds for the semidiscrete solution.
\begin{lemma}\label{flm3.1}
	Let $w_0\in L^2(\Omega)$.Then, there holds
	\begin{align*}
	\norm{w_h(t)}^2&+\beta e^{-2\alpha t}\int_{0}^{t}e^{2\alpha s}\big(\norm{\nabla w_h(s)}^2+\norm{w_h(s)}^2_{L^2(\partial\Omega)}\big)ds+\frac{1}{3c_0}e^{-2\alpha t}\int_{0}^{t}e^{2\alpha s}\norm{w_h(s)}^4_{L^4(\partial\Omega)} ds\\
	&\leq e^{-2\alpha t}\norm{w_{0h}}^2.
	\end{align*}
\end{lemma}
\subsection{Error estimates}
Define an auxiliary projection $\tilde w_h\in V_h$ of $w$ through the following form
\begin{equation}\label{feq3.1}
\Big(\nabla(w-\tilde w_h),\nabla\chi\Big)+\lambda\Big(w-\tilde w_h,\chi\Big)=0\quad \forall~ \chi\in V_h,
\end{equation}
where $\lambda\geq 1$ is some fixed positive number. For a given $w,$ the existence of a unique $\tilde w_h$ follows by the Lax-Milgram Lemma. 
Let $\eta:=w-\tilde w_h$ be the error involved in the auxiliary projection. Then, the following error estimates hold: 
\begin{align}\label{feq3.2}
\norm{\eta}_j&\leq C h^{\min(2,m)-j}\;\norm{w}_m, \;\text{and}\notag\\
\norm{\eta _t}_j&\leq C h^{\min(2,m)-j}\norm{w_t}_m,\;\;
j=0,1 \;\mbox{ and }\; m=1,2.
\end{align}
For a proof, we refer to Thom\'{e}e \cite{thomee}, \cite{ciarlet}.
Following Lemma \ref{x1} is needed to establish  error estimates.
\begin{lemma}\label{x1}
Let $F\in H^{3/2+\epsilon}(\Omega)$, for some $\epsilon>0$, and $G\in H^{1/2}(\partial\Omega).$ Then $FG\in H^{1/2}(\partial\Omega)$ and
$$\norm{FG}_{H^{1/2}(\partial\Omega)}\leq C\norm{F}_{H^{3/2+\epsilon}(\Omega)}\norm{G}_{H^{1/2}(\partial\Omega)}.$$
\end{lemma}
\begin{proof}
For a proof see \cite{dt73}.
\end{proof}
In addition, for proving error estimates for state variable and feedback controllers, we need the following estimates of $\eta$ and $\eta_t$ at boundary.
\begin{lemma}\label{flm3.6}
	For smooth $\partial\Omega,$ there holds	
	\begin{align*}
&\norm{\eta}_{L^2(\partial\Omega)}\leq Ch^{3/2}\norm{w}_2, \quad	\norm{\eta}_{H^{-1/2}(\partial\Omega)}\leq Ch^2\norm{w}_2, \quad \norm{\eta_t}_{H^{-1/2}(\partial\Omega)}\leq Ch^2\norm{w_t}_2,\\
	&\norm{\eta}_{L^q(\partial\Omega)}\leq Ch\norm{w}_2, \quad\text{and} \quad \norm{\eta_t}_{L^q(\partial\Omega)}\leq Ch\norm{w_t}_{2},\;\; 2\leq q <\infty.
	\end{align*}
\end{lemma}
\begin{proof}
	Consider an auxiliary function $\phi$ satisfying the following problem
	\begin{align}
	-\Delta \phi+\lambda \phi&=0 \quad \text{in} \quad\Omega\label{feq3.3},\\\frac{\partial \phi}{\partial\nu}&=\eta \quad \text{on} \quad\partial\Omega,\notag
	\end{align}
	with $\norm{\phi}_2\leq C\norm{\eta}_{H^\frac{1}{2}(\partial\Omega)}$. For a proof of this regularity result see \cite{lm68}.\\
	Take the inner product between \eqref{feq3.3} and $\eta$ to obtain
	\begin{align*}
	\norm{\eta}^2_{L^2(\partial\Omega)}=(\nabla \phi,\nabla\eta)+\lambda(\phi,\eta)&=(\nabla \phi-\nabla\tilde{\phi_h},\nabla\eta)+\lambda(\phi-\tilde{\phi_h},\eta)\\
	&\leq ch^2\norm{\phi}_2\norm{w}_2+Ch^4\norm{\phi}_2\norm{w}_2\leq Ch^2\norm{w}_2\norm{\eta}_{H^\frac{1}{2}(\partial\Omega)}.
	\end{align*}
	Using the Trace inequality \eqref{1.31} for $s=1/2$, we arrive at
	\begin{align*}
	\norm{\eta}^2_{L^2(\partial\Omega)}\leq Ch^2\norm{w}^2_2\norm{\eta}_{H^1(\Omega)}\leq Ch^3\norm{w}^2_2.
	\end{align*}	
	Hence, $\norm{\eta}_{L^2(\partial\Omega)}\leq Ch^\frac{3}{2}\norm{w}_2$.\\
	The idea for showing estimate $\norm{\eta}_{H^{-1/2}(\partial\Omega)}$ using variant of the  Aubin-Nitsche technique can be found in \cite{dt73}. For completeness, we provide a brief proof here.
	Let $\beta=\beta(t)$ be the solution of 
	\begin{equation}\label{xl3.1}
	(\nabla \beta,\nabla\chi)+\lambda(\beta,\chi)=\langle\delta,\chi\rangle_{\partial \Omega},
	\end{equation}
	where $\delta\in H^{1/2}(\partial\Omega)$ is such that
	$$\norm{\delta}_{H^{1/2}(\partial\Omega)}=\norm{\eta}_{H^{-1/2}(\partial\Omega)}, \quad \langle\delta,\eta\rangle_{\partial \Omega}=\norm{\eta}^2_{H^{-1/2}(\partial\Omega)},$$
	where the existence of $\delta$ follows from the Hahn-Banach theorem.
	Set $\chi=\eta$ and use \eqref{feq3.1} to obtain
	\begin{align*}
\norm{\eta}^2_{H^{-1/2}(\partial\Omega)}&=(\nabla \eta, \nabla \beta)+\lambda(\eta,\beta)\\
&=(\nabla \eta, \nabla (\beta-\phi))+\lambda(\eta,\beta-\phi)\leq C\norm{\eta}_{1}\norm{\beta-\phi}_1 \quad \forall \phi\in V_h.
	\end{align*}
	Therefore using $\inf_{\chi \in V_h}\norm{v-\chi}_i\leq Ch^{2-i}\norm{v}_2, \quad i=0,1$ and \eqref{feq3.3}, it follows that
	\begin{align*}
\norm{\eta}^2_{H^{-1/2}(\partial\Omega)}&\leq Ch\norm{\eta}_1\norm{\beta}_2\leq Ch\norm{\eta}_1\norm{\delta}_{H^{1/2}(\partial\Omega)}\leq Ch\norm{\eta}_1\norm{\eta}_{H^{-1/2}(\partial\Omega)}.
	\end{align*}
	Hence
$\norm{\eta}_{H^{-1/2}(\partial\Omega)}\leq Ch^2\norm{w}_2.$\\
	Consider an auxiliary function $\phi$ satisfying the following problem
\begin{align}
-\Delta \phi+\lambda \phi&=0 \quad \text{in} \quad\Omega\label{feq3.4},\\\frac{\partial \phi}{\partial\nu}&=\eta_t \quad \text{on} \quad\partial\Omega,\notag
\end{align}
where $\norm{\phi}_2\leq C\norm{\eta_t}_{H^\frac{1}{2}(\partial\Omega)}$.  For a proof of this regularity result see \cite{lm68}.\\
Similarly we can show that
$\norm{\eta_t}_{H^{-1/2}(\partial\Omega)}\leq Ch^2\norm{w_t}_2.$

Using \eqref{1.30}, it follows  for $q \in [2,\infty)$ that
	\begin{align*}
	\norm{\eta}_{L^q(\partial \Omega)}\leq C\norm{\eta}_1\leq Ch\norm{w}_2,
	\end{align*} 
	and
	\begin{equation*}
	\norm{\eta_t}_{L^q(\partial \Omega)}\leq C\norm{\eta_t}_1\leq Ch\norm{w_t}_2.
	\end{equation*}
	This completes the proof.
\end{proof}
With $e:=w-w_h,$ decompose $e:=(w-\tilde w_h)- (w_h-\tilde w_h)=:\eta-\theta,$ where $\eta=w-\tilde w_h$ and $\theta=w_h-\tilde w_h$.\\
Since estimates of $\eta$ are known from \eqref{feq3.2} and Lemma \ref{flm3.6}, it is sufficient to estimate $\theta$.
Subtracting the weak formulation \eqref{feq1.10} from \eqref{feqn3.1} and a use of \eqref{feq3.1} yields
\begin{align}\label{feqn3.5}
(\theta_t,\chi)&+\nu(\nabla \theta, \nabla \chi)+\int_{\partial\Omega} \Big( 2 (c_0+w_d) \theta +\frac{2}{9c_0} \theta^3+ \frac{2}{3c_0} w_h^2 \theta\Big) \chi\,\;d\Gamma \notag \\
&=(\eta_t -\nu\lambda  \eta,\chi) +w_d\big(\nabla(\eta-\theta)\cdot{\bf{1}},\chi\big)+\int_{\partial\Omega} 2 (c_0+w_d)\;\eta\;\chi\;d\Gamma\notag\\
&\qquad
+\big((\eta-\theta)\nabla w\cdot{\bf{1}}+w_h(\nabla \eta-\nabla \theta)\cdot{\bf{1}},\chi\big) +\frac{2}{9c_0}\int_{\partial\Omega}\big(\eta^3 +3w\eta(w-\eta)+3 w_h \theta^2 \big) \chi \;d\Gamma\notag\\
&\qquad = \sum_{i=1}^{5}I_i(\chi).
\end{align}
Hence forward, we like to restrict the range of $\alpha$, that is,
$$  0 \leq \alpha \leq \frac{1}{C_F}\min\{ 3\nu/4, ((c_0/2)+ w_d)\}.$$
For such $\alpha$, the results of section 3 are also valid.

In the following theorem, we estimate $\norm{\theta(t)}$.
\begin{theorem}\label{fthm3.1}
	Let $w_0\in H^3(\Omega)$. Then, there exists a positive constant $C=C(\norm{w_0}_3)$ such that 
	there holds
	\begin{align*}
	\norm{\theta(t)}^2&+ \beta_1\;e^{-2\alpha t}\int_{0}^{t}e^{2\alpha s}\Big(\norm{ \nabla\theta(s)}^2+ \norm{\theta(s)}^2_{L^2(\partial\Omega)}
	+\norm{\theta(s)}^4_{L^4(\partial\Omega)}
	\Big)\;ds\\
	&\qquad\qquad\leq C\Big(\norm{w_0}_3\Big)\exp\Big(\norm{w_0}_2\Big)h^4e^{-2\alpha t},
	\end{align*}
	where $\beta_1= \min\Big(\big(\frac{3\nu}{2}-2\alpha C_F\big), \big((c_0+2w_d)- 2\alpha C_F\big), \frac{1}{27 c_0} \Big)>0$.
\end{theorem}
\begin{proof}
	Set $\chi=\theta$ in \eqref{feqn3.5} to obtain
	\begin{align}\label{feq3.6}
	\frac{1}{2}\frac{d}{dt}\norm{\theta}^2+\nu\norm{\nabla\theta}^2 +2 (c_0+ w_d) \|\theta\|^2_{L^2(\Omega)} + \frac{2}{9c_0}  \|\theta\|^4_{L^4(\Omega)} 
	+ \frac{2}{3c_0} \int_{\partial\Omega} w_h^2\;\theta^2\;d\Gamma =\sum_{i=1}^{5}I_i(\theta).
	\end{align}
	The first term $I_1(\theta)$ on the right hand side of \eqref{feq3.6} is bounded by 
	$$I_1(\theta)=(\eta_t -\nu\lambda \eta,\theta)\leq C\big(\norm{\eta}^2+\norm{\eta_t}^2\big)+\frac{\epsilon}{16}\norm{\theta}^2,$$
	where $\epsilon>0$ is a positive number which we choose later.
	For the second term $I_2(\theta)$ on the right hand side of \eqref{feq3.6}, a use
	of the Cauchy-Schwarz inequality with Young's inequality and $\norm{\theta}_{H^{1/2}(\partial\Omega)}\leq C\norm{\theta}_1$ yields
	\begin{align*}
	I_2(\theta)&=w_d\Big(\nabla(\eta-\theta)\cdot{\bf{1}},\theta\Big)\\
	&=-w_d\big(\eta,\nabla\theta\cdot{\bf{1}}\big)+w_d\sum_{i=1}^{2}\int_{\partial\Omega}\eta \nu_i\theta d\Gamma-\frac{w_d}{2}\sum_{i=1}^{2}\int_{\partial\Omega}\theta^2 \nu_i d\Gamma\\
&\leq C\norm{\eta}\norm{\nabla\theta}+C\norm{\eta}_{H^{-1/2}(\partial\Omega)}\norm{\theta}_{H^{1/2}(\partial\Omega)}+w_d\norm{\theta}^2_{L^2(\partial\Omega)}\\
	&\leq \frac{\nu}{24}\norm{\nabla \theta}^2+\frac{\epsilon}{16}\norm{\theta}^2+C\norm{\eta}^2+w_d\norm{\theta}^2_{L^2(\partial\Omega)}+C\norm{\eta}^2_{H^{-1/2}(\partial\Omega)}.
	\end{align*}
	The third term $I_3(\theta)$ on the right hand side is bounded by
	\begin{align*}
	2(c_0+ w_d)\langle\eta,\theta\rangle_{(\partial\Omega)}&\leq C\;\norm{\eta}_{H^{-1/2}(\partial\Omega)}\;\norm{\theta}_{H^{1/2}(\partial\Omega)} \leq C \norm{\eta}_{H^{-1/2}(\partial\Omega)}\; (\norm{\nabla \theta} + \norm{\theta}) \\
&\hspace{1.5cm} \leq  \frac{\nu}{24}\norm{\nabla \theta}^2+\frac{\epsilon}{16}\norm{\theta}^2+C\norm{\eta}^2_{H^{-1/2}(\partial\Omega)}.
\end{align*}
	For the fourth term $I_4(\theta),$ first we use the Gagliardo-Nirenberg inequality  for $H^1$ function to  bound the following sub-terms as
	\begin{align*}
	\big((\eta-\theta)&\nabla w\cdot{\bf{1}},\theta\big) 
	\leq C\norm{\eta}\norm{\nabla w}_{L^4}\norm{\theta}_{L^4}+C\norm{\theta}\norm{\theta}_{L^4}\norm{\nabla w}_{L^4} 
	\\
	&\leq \frac{\nu}{48}\norm{\nabla \theta}^2+\frac{\epsilon}{16}\norm{\theta}^2+C\norm{\theta}^2  \norm{w}^2_2 
	+C\norm{\eta}^2\Big(1+\norm{w}^2+\norm{\Delta w}^2\Big)
	\end{align*}
	and  apply  $w_h= \theta + \tilde{w}_h$  with integration by parts  and $ \|\tilde{w}_h\|_{L^{\infty}} \leq C \|w\|_2$ to obtain a bound
\begin{align*}
	-(w_h\nabla \theta\cdot{\bf{1}},\theta)&= -( \theta \nabla \theta\cdot{\bf{1}},\theta)-(\tilde{w}_h\nabla \theta\cdot{\bf{1}},\theta)
	=-\frac{1}{3} \sum_{j=1}^2 \int_{\partial \Omega} \theta^3\; n_j \;d\Gamma-(\tilde{w}_h\nabla \theta\cdot{\bf{1}},\theta)\\
	&\leq \frac{\sqrt{2}}{3} \int_{\partial \Omega} |\theta|^3 \;d\Gamma 
	+ \|\tilde{w}_h\|_{L^{\infty}} \norm{\nabla \theta}\;\norm {\theta}\\
	&\leq \frac{3c_0}{2} \int_{\partial \Omega} |\theta|^2 \;d\Gamma +  \frac{1}{27 c_0} \int_{\partial \Omega} |\theta|^4 \;d\Gamma  +\frac{\nu}{48} \norm{\nabla \theta}^2
	+C\norm{\theta}^2  \norm{w}^2_2.
	\end{align*}
Using the Sobolev inequality $\norm{\theta}_{L^4}  \leq C \norm{\theta}_1$ and $\norm{\tilde{w}_h\theta}_{H^{1/2}(\partial\Omega)}\leq C \norm{w}_2 \; \norm{\theta}_1$, the other sub-term in $I_4(\theta)$ can be bounded by
	\begin{align*}
	\big(w_h\nabla\eta\cdot{\bf{1}},\theta\big)
	&=\big(\theta\nabla\eta\cdot{\bf{1}},\theta\big)-\big(\tilde{w}_h\nabla\theta\cdot{\bf{1}},\eta\big)-\big(\eta \nabla \tilde{w}_h\cdot{\bf{1}},\theta\big)+\sum_{i=1}^{2}\int_{\partial\Omega}\tilde{w}_h\eta\nu_i\theta\; d\Gamma\\
	&\leq \norm{\theta}\norm{\nabla \eta}_{L^4} \;\norm{\theta}_{L^4} + \norm{\eta}\norm{\nabla \theta}\norm{\tilde{w}_h}_{L^\infty}+ \norm{\eta}\norm{\nabla \tilde{w}_h}_{L^4}\norm{\theta}_{L^4}\\
	&\quad+\norm{\eta}_{H^{-1/2}(\partial\Omega)}\norm{\tilde{w}_h\theta}_{H^{1/2}(\partial\Omega)}\\
	&\leq \frac{\nu}{24}\norm{\nabla \theta}^2+\frac{\epsilon}{16}\norm{\theta}^2+C\; \big(\norm{w}^2_2\;\norm{\eta}^2  + \norm{\eta}^2_{H^{-1/2}(\partial\Omega)}\big)
	+C\norm{\theta}^2\big(\norm{w}^2+\norm{\nabla \eta}^2_{L^4}\big). 
	\end{align*}
	For $I_5(\theta),$ we note that
	\begin{align*}
	\frac{2}{9c_0}\int_{\partial\Omega}\eta^3\theta d\Gamma &\leq  C\; \norm{\eta}^3_{L^6(\partial\Omega)}\norm{\theta}_{L^2(\partial\Omega)}
	\leq  C \;\norm{\eta}^3_{L^6(\partial\Omega)}\;(\norm{\nabla \theta} + \norm{\theta})\\
	&\qquad \leq \frac{\nu}{48}\norm{\nabla \theta}^2 + \frac{\epsilon}{16} \norm{\theta}^2 + C \;\norm{\eta}^6_{L^6(\partial\Omega)},
	\end{align*}
	\begin{align*}
	\frac{2}{9c_0}\int_{\partial\Omega}3w^2\eta\theta d\Gamma&\leq
	C\norm{\eta}_{H^{-1/2}(\partial\Omega)}\norm{w^2\theta}_{H^{1/2}(\partial\Omega)}\\
	&\leq	C\norm{\eta}_{H^{-1/2}(\partial\Omega)}\norm{w}_2\norm{w\theta}_{H^{1/2}(\partial\Omega)}\\
	&\leq 	C\norm{\eta}_{H^{-1/2}(\partial\Omega)}\norm{w}^2_2\norm{\theta}_1\\
	&\leq \frac{\epsilon}{16}\norm{\theta}^2+\frac{\nu}{48}\norm{\nabla \theta}^2+C\norm{\eta}^2_{H^{-1/2}(\partial\Omega)}\norm{w}^4_2,
	\end{align*}
	\begin{align*}
	\frac{2}{9c_0}\int_{\partial\Omega}3w\eta^2\theta d\Gamma\leq \frac{\epsilon}{16}\norm{\theta}^2+\frac{\nu}{24}\norm{\nabla \theta}^2+C\norm{w}^2_{L^4(\partial\Omega)}\norm{\eta}^4_{L^4(\partial\Omega)},
	\end{align*}
	and
	\begin{align*}
	\frac{2}{9c_0}\int_{\partial\Omega}3w_h\theta^3 d\Gamma\leq\frac{2}{3c_0}\Bigg(\int_{\partial\Omega} w_h^2\theta^2 d\Gamma+\frac{1}{4}\int_{\partial\Omega}\theta^4 d\Gamma\bigg).
	\end{align*}
	Finally, using Lemmas \ref{flm1}-\ref{flm4}, \ref{flm3.1} and \ref{flm3.6}, we arrive from \eqref{feq3.6} at
	\begin{align}
	\frac{d}{dt}\norm{\theta}^2&+\frac{3}{2}\nu\norm{\nabla \theta}^2+(c_0+2w_d)\norm{\theta}^2_{L^2(\partial\Omega)}+\frac{1}{27c_0}\norm{\theta}^4_{L^4(\partial\Omega)}\notag\\
	&\leq \epsilon\norm{\theta}^2+C\norm{\eta}^2(1+\norm{w}^2+\norm{\Delta w}^2)+C\norm{\eta}^6_{L^6(\partial\Omega)}+C\norm{\theta}^2(1+h^2)\norm{w}^2_2\notag\\
	&\qquad+C\Big(\norm{\eta}^4_{L^4(\partial\Omega)}+\norm{\eta}^2_{H^{-1/2}(\partial\Omega)}\Big)\Big(1+\norm{w}^2_2+\norm{w}^4_{L^4(\partial\Omega)}\Big)\label{feq3.10}.
	\end{align}
	Multiply \eqref{feq3.10} by $e^{2\alpha t}$ and use Friedrichs's inequality
	$$-2\alpha e^{2\alpha t}\norm{\theta}^2\geq -2\alpha C_F e^{2\alpha t}\norm{\nabla\theta}^2-2\alpha C_Fe^{2\alpha t}\norm{\theta}^2_{L^2(\partial\Omega)}.$$ Then
	 a use of Lemmas \ref{flm1}, \ref{flm2},  and \ref{flm3.6} in \eqref{feq3.10} yields
	\begin{align*}
	\frac{d}{dt}&\big(\norm{e^{\alpha t}\theta}^2\big)+e^{2\alpha t}\Bigg(\Big(\frac{3\nu}{2}-2\alpha C_F\Big)\norm{ \nabla\theta}^2
	+\Big((c_0+2w_d)-2\alpha C_F\Big)\norm{ \theta}^2_{L^2(\partial\Omega)}+\frac{1}{27c_0}\norm{\theta}^4_{L^4(\partial\Omega)}
	\Bigg)\\
	&\leq Ce^{2\alpha t}\Big(\norm{\eta}^2+\norm{\eta_t}^2 +\norm{\eta}^2_{H^{-1/2}(\partial\Omega)}+C\norm{\eta}^4_{L^4(\partial\Omega)}+\norm{\eta}^6_{L^6(\partial\Omega)} \Big)
	+Ce^{2\alpha t}\norm{\theta}^2\;\big((1+h^2)\norm{w}^2_2\big) \\
	&\quad
	+\epsilon C_Fe^{2\alpha t}\Big(\norm{ \nabla\theta}^2+\norm{\theta}^2_{L^2(\partial\Omega)}\Big).
	\end{align*}
	Integrate the above inequality from $0$ to $t$ and choose $\epsilon=\frac{\beta_1}{2C_F}$. Then use the Gr\"onwall's inequality to obtain
	\begin{align*}
	&\norm{e^{\alpha t}\theta(t)}^2+ \beta_1\int_{0}^{t}e^{2\alpha s}\Big(\norm{ \nabla\theta(s)}^2+\norm{\theta(s)}^2_{L^2(\partial\Omega)}+\norm{\theta(s)}^4_{L^4(\partial\Omega)}\Big)\;ds\\
	&\quad\leq Ch^4\Big(\int_{0}^{t}\big(\norm{w(s)}^2_2+\norm{w_t(s)}^2_2\big)\;ds\Big)\exp\Bigg(\int_{0}^{t}\Big((1+h^2)\norm{w(s)}^2_2\Big)\;ds\Bigg).
	\end{align*}
	A use of Lemmas \ref{flm1}-\ref{flm5}, and \ref{flm3.1}  to the above inequality with a multiplication of $e^{-2\alpha t}$ completes  the proof.
\end{proof}
\begin{remark}
As a consequence of Theorem ~\ref{fthm3.1}, we use inverse property \eqref{inverse-property} to arrive for $p\in [2,\infty]$ at
\begin{align}\label{estimate-w-h-max}
\norm{w_h(t)}_{L^p(\Omega)} &\leq \norm{\tilde{w}_h(t)}_{L^p(\Omega)} + \norm{\theta(t)}_{L^p(\Omega)}\nonumber\\
&\leq  C\; \norm{w(t)}_2 + C\; h^{2(\frac{1}{p}-\frac{1}{2})}\; \norm{\theta(t)} \leq C e^{-\alpha t} \leq C.
\end{align}
\end{remark}
\begin{theorem}\label{fthm3.2}
	Let $w_0\in H^3(\Omega)$. Then, there is a positive constant $C$ independent of $h$ such that
	\begin{align*}
	\nu \norm{\nabla \theta(t)}^2+&2 (c_0+ w_d)\norm{\theta(t)}^2_{L^2(\partial\Omega)}+\frac{1}{9c_0}\norm{\theta(t)}^4_{L^4(\partial\Omega)}+e^{-2\alpha t}\int_{0}^{t}e^{2\alpha s}\norm{\theta_t(s)}^2\;ds\\
	&\leq C\Big(\norm{w_0}_3\Big)\exp\Big(C\norm{w_0}_2\Big)h^4e^{-2\alpha t}.
	\end{align*}
\end{theorem}
\begin{proof}
	Set $\chi=\theta_t$ in \eqref{feqn3.5} to obtain
	\begin{align}
	\norm{\theta_t}^2+\frac{1}{2}\frac{d}{dt}\Big(\nu\norm{\nabla\theta}^2 &+ 2(c_0+w_d) \norm{\theta}^2_{L^2(\partial \Omega)} + \frac{1}{9c_0} \norm{\theta}^4_{L^4(\partial\Omega)} + \frac{2}{3c_0} \norm{w_h \theta}^2_{L^2(\partial \Omega)}\Big)\nonumber\\
	&\quad \quad=\sum_{i=1}^{4}I_i(\theta_t) + \Big( I_5(\theta_t) +\frac{2}{3c_0} \big \langle w_h\;w_{ht}\; \theta,\theta \big \rangle_{\partial\Omega}\Big)\label{feq5.1}.
	\end{align}
	The first term $I_1(\theta_t)$ on the right hand side of \eqref{feqn3.5} is bounded by
	\begin{equation*}
	I_1(\theta_t)=(\eta_t-\lambda\nu \eta,\theta_t) \leq \frac{1}{6}\norm{\theta_t}^2+C\Big(\norm{\eta}^2+\norm{\eta_t}^2\Big).
	\end{equation*}
	The second term $I_2(\theta_t)$ on the right hand side of \eqref{feqn3.5} can be rewritten as 
	\begin{align*}
	I_2(\theta_t)&=w_d\Big((\nabla \eta-\nabla\theta)\cdot {\bf{1}},\theta_t\Big)\\
	&=-w_d\frac{d}{dt}\Big(\eta,\nabla\theta\cdot {\bf{1}}\Big)+w_d(\eta_t,\nabla\theta\cdot {\bf{1}})+w_d\frac{d}{dt}\Big(\sum_{i=1}^{2}\int_{\partial\Omega}\eta\nu_i\theta \;d\Gamma\Big)\\
	&\quad -w_d\sum_{i=1}^{2}\int_{\partial\Omega}\eta_t\nu_i\theta\;d\Gamma-w_d(\nabla\theta\cdot{\bf{1}},\theta_t),
	\end{align*}
	and hence, we get
	\begin{align*}
	I_2(\theta_t)&=w_d\Big((\nabla \eta-\nabla\theta)\cdot {\bf{1}},\theta_t\Big)\\
	&\leq -w_d\frac{d}{dt}\Big(\eta,\nabla\theta\cdot {\bf{1}}\Big)+
	w_d\frac{d}{dt}\Big(\sum_{i=1}^{2}\int_{\partial\Omega}\eta\nu_i\theta d\Gamma\Big)\\
	&\quad+ \frac{1}{12}\norm{\theta_t}^2+C\Big(\norm{\theta}^2+\norm{\nabla\theta}^2\Big)+C\Big(\norm{\eta_t}^2+\norm{\eta_t}^2_{H^{-1/2}(\partial\Omega)}\Big).
	\end{align*}
	The third term $I_3(\theta_t)$ on the right hand side of \eqref{feqn3.5} is bounded by
	\begin{align*}
	I_3(\theta_t)&=\Big \langle 2(c_0+w_d)\eta,\theta_t\Big \rangle_{\partial\Omega} = \frac{d}{dt}\Big\langle 2(c_0+w_d)\eta,\theta\Big\rangle- 2(c_0+w_d)\langle\eta_t,\theta\rangle\\
	&\leq  
	\frac{d}{dt}\Big\langle 2(c_0+ w_d)\eta,\theta\Big\rangle
	+C\Big(\norm{\theta}^2+\norm{\nabla\theta}^2\Big)+C\norm{\eta_t}^2_{H^{-1/2}(\partial\Omega)}.
	\end{align*}
	For the fourth term $I_4(\theta_t)$ on the right hand side of \eqref{feqn3.5}, first we rewrite the sub terms as
	\begin{align*}
	(\eta\nabla w\cdot{\bf{1}},\theta_t)=\frac{d}{dt}\Big((\eta\nabla w\cdot {\bf{1}},\theta)\Big)-(\eta_t\nabla w\cdot {\bf{1}},\theta)-\Big(\eta(\nabla w\cdot {\bf{1}})_t,\theta\Big),
	\end{align*}
	and bound it using the Gagliardo-Nirenberg inequality for $H^2$ function
	\begin{align*}
	(\eta\nabla w\cdot{\bf{1}},\theta_t)&\leq\frac{d}{dt}\Big((\eta\nabla w\cdot {\bf{1}},\theta)\Big)+ C\norm{\eta_t}^2\Big(1+\norm{w}^2_2\Big)+C\norm{\theta}^2\Big(1+\norm{\Delta w}^2+\norm{\Delta w_t}^2\Big)\\
	&\qquad+C\norm{\nabla \theta}^2\Big(\norm{w}^2+\norm{w_t}^2\Big)+C\norm{\eta}^2\norm{\Delta w_t}^2.
	\end{align*}
	For the other sub-term of $I_4(\theta_t)$, apply $w_h= \theta + \tilde{w}_h$ with and then  use integration by parts    to obtain
	\begin{align*}
	(w_h\nabla \eta\cdot{\bf{1}},\theta_t)&=(\theta\nabla \eta\cdot{\bf{1}},\theta_t)
	-\frac{d}{dt} \Big((\eta\nabla\theta\cdot {\bf{1}}, \tilde{w}_h)+\big(\eta(\nabla \tilde{w}_h \cdot {\bf{1}}),\theta\big)-\sum_{i=1}^{2}\int_{\partial\Omega}\eta \tilde{w}_h\nu_i\theta\; d\Gamma\Big)\\
	&\qquad +\Big((\eta \tilde{w}_h)_t,\nabla\theta\cdot{\bf{1}}\Big)+\Big((\eta\nabla \tilde{w}_h\cdot{\bf{1}})_t,\theta\Big)-\sum_{i=1}^{2}\int_{\partial\Omega}(\eta \tilde{w}_h)_t\nu_i\theta \;d\Gamma,
	\end{align*}
and hence, from $\norm{\theta}_{L^4(\Omega)} \leq C (\norm{\nabla \theta} + \norm{\theta})$ and $\norm{\tilde{w}_h}_{L^{\infty}}\leq C (1+h)\;  \norm{w}_2,$  (see, Ciarlet \cite[page 168]{ciarlet}), it follows that
	\begin{align*}
	\Big(w_h\nabla\eta \cdot{\bf{1}},\theta_t\Big)&\leq -\frac{d}{dt}\Big((\eta\nabla\theta\cdot {\bf{1}}, \tilde{w}_h)+\big(\eta(\nabla \tilde{w}_h \cdot {\bf{1}}),\theta\big)-\sum_{i=1}^{2}\int_{\partial\Omega}\eta \tilde{w}_h\nu_i\theta\; d\Gamma\Big)\\
	&\quad+ C\norm{\eta_t}^2\Big(1+\norm{w}^2_2\Big)+C\norm{\eta}^2\Big(1+\norm{w_{t}}^2_2\Big)\\
	&\quad+C \big(\norm{\theta}^2 +\norm{\nabla\theta}^2\big)\; \Big(\norm{\nabla \eta}_{L^4}^2+\norm{\nabla \eta_t}^2_{L^4}+\norm{w}^2_2+\norm{w_{t}}^2_2\Big)\\
	&\quad+C\Big(\norm{\eta}^2_{H^{-1/2}(\partial\Omega)}+\norm{\eta_t}^2_{H^{-1/2}(\partial\Omega)}\Big) + \frac{1}{12} \norm{\theta_t}^2.
	\end{align*}
	For the  remaining other two sub-terms  of $I_4(\theta_t),$   a use of $\norm{\theta(t)}_{L^{\infty}} \leq C\;h  e^{-\alpha t}\leq C$  now leads to
	\begin{align*}
	-\Big(\theta(\nabla w\cdot{\bf{1}}),\theta_t\Big)&-\Big(w_h(\nabla\theta\cdot{\bf{1}}),\theta_t\Big)= -\Big(\theta(\nabla w\cdot{\bf{1}}),\theta_t\Big)-\Big( (\theta + \tilde{w}_h)\;\nabla\theta\cdot{\bf{1}},\theta_t\Big)\\
	&\leq \Big(\norm{\nabla w}_{L^4}\; \norm{\theta}_{L^4} + (\norm{\theta}_{L^{\infty} }+ \norm{\tilde{w}_h}_{L^{\infty}}) \norm{\nabla \theta}\Big)\;\norm{\theta_t}\\
	&\leq \frac{1}{12}\norm{\theta_t}^2+C\;\norm{\theta}^2 \;\norm{w}_2^2+ \norm{\nabla \theta}^2\;\big(\norm{w}^2_2+ 1\big).
	\end{align*}
	For the last term 
	on the right hand side of \eqref{feqn3.5}, the first sub-term is bounded by
	\begin{align*}
	\frac{2}{9c_0}\int_{\partial\Omega}\eta^3\theta_t \;d\Gamma&=\frac{2}{9c_0}\frac{d}{dt}\Big(\int_{\partial\Omega}\eta^3\theta \;d\Gamma\Big)-\frac{2}{3c_0}\int_{\partial\Omega}\eta^2\eta_t\theta\;d\Gamma\\
	&\leq \frac{2}{9c_0}\frac{d}{dt}\Big(\int_{\partial\Omega}\eta^3\theta \;d\Gamma\Big)+C\Big(\norm{\eta}^4_{L^4(\partial\Omega)}+\norm{\eta_t}^2_{L^4(\partial\Omega)}\norm{\theta}^2_1\Big).
	\end{align*}
	Similarly, the other sub-terms are bounded by
	\begin{align*}
	\frac{2}{9c_0}\int_{\partial\Omega}3w^2\eta\theta_t\;d\Gamma&\leq\frac{2}{3c_0}\frac{d}{dt}\Big(\int_{\partial\Omega}w^2\eta \theta\;d\Gamma\Big)+C\norm{\eta_t}_{H^{-1/2}(\partial\Omega)}\norm{w}^2_2\norm{\theta}_1\\
	&+C\norm{\eta}_{H^{-1/2}(\partial\Omega)}\norm{w}_2\norm{w_t}_2\norm{\theta}_1\\&\leq \frac{2}{3c_0}\frac{d}{dt}\Big(\int_{\partial\Omega}w^2\eta \theta\;d\Gamma\Big)+C\norm{w}^2_2\Big(\norm{\theta}^2+\norm{\nabla\theta}^2\Big)\\
	&\quad+C\norm{\eta_t}^2_{H^{-1/2}(\partial\Omega)}+C\norm{\eta}^2_{H^{-1/2}(\partial\Omega)}\norm{w_t}^2_2,
	\end{align*}
	\begin{align*}
	-\frac{2}{9c_0}3\int_{\partial\Omega}w\eta^2\theta_t\;d\Gamma&\leq-\frac{2}{3c_0}\frac{d}{dt}\Big(\int_{\partial\Omega}w\eta^2\theta\;d\Gamma\Big)+C\norm{\eta}^4_{L^4(\partial\Omega)}\\
	&\qquad+C\norm{\eta}^2_{L^4(\partial\Omega)}\norm{\eta_t}^2_{L^4(\partial\Omega)}+C\norm{\theta}^2_1\Big(\norm{w}^2_1+\norm{w_t}^2_1\Big).
	\end{align*}
Trace inequality for fractional Sobolev norm  yields
$$\norm{v_h}_{L^\infty(\partial K)} \leq C( \norm{v_h}_{L^{\infty}(K)} + h_{K}^{\delta} \norm{v_h}_{W^{\delta,\infty}(K)}).$$
Using inverse inequality $\norm{v_h}_{W^{\delta,\infty}(K)} \leq C h^{-\delta} \norm{v_h}_{L^{\infty}(K)}$,
we obtain $\norm{v_h}_{L^\infty(\partial K)} \leq C \norm{v_h}_{L^{\infty}(K)}$.	
		
A use of $\|\theta(t)\|_{L^{\infty}} \leq C \;h e^{-\alpha t}$ with $w_{ht}= \theta_t + \tilde{w}_{ht}$, $\norm{v_h}_{L^\infty(\partial K\cap \partial \Omega)}\leq \norm{v_h}_{L^\infty(\Omega)}$, 
and inverse inequality $\|\theta\|_1\leq C h^{-1} \|\theta\|$ yields
	\begin{align*}
	\frac{2}{3c_0} \int_{\partial\Omega} \Big( w_h \theta^2 \;\theta_t &+ w_h\;w_{ht}\;\theta^2\Big)\;d\Gamma = \frac{2}{3c_0} \int_{\partial\Omega} \Big( 2 w_h \theta^2 \;\theta_t + w_h\;\tilde{w}_{ht}\;\theta^2\Big)\;d\Gamma\\
	&\leq C\;\|w_h\|_{L^{\infty}}\Big(\norm{\theta}_{L^{\infty}} \;\norm{\theta}_{L^2(\partial\Omega)}\;\norm{\theta_t}_{L^2(\partial\Omega)} + \big(\norm{w_h}^2_{L^{\infty}} +\norm{\tilde{w}_{ht}}^2_{L^{\infty}}\big)\;\norm{\theta}^2_{L^2(\partial\Omega)}\Big)\\
	& \leq C\;\Big( h\;\norm{w}_2\;\norm{\theta}_{L^2(\partial\Omega)}\;\norm{\theta_t}_1 + \big(\norm{w_t}_2^2+\norm{w}_2^2\big)\; \norm{\theta}^2_{L^2(\partial\Omega)}\Big)\\
	& \leq C\; \big(\norm{w_t}_2^2+\norm{w}_2^2\big)\;\norm{\theta}^2_{L^2(\partial\Omega)} + \frac{1}{12} \norm{\theta_t}^2.
	\end{align*}
	Substitute all the above estimates in \eqref{feq5.1}, apply kickback argument   and then multiply the resulting ones by $e^{2\alpha t}$  with using Lemmas \ref{flm1}, \ref{flm4}, \ref{flm3.1}  and \ref{x1}  to obtain after setting 
	$$\||\theta(t)\|| := \Big(\nu\norm{\nabla\theta}^2 + 2(c_0+w_d) \norm{\theta}^2_{L^2(\partial \Omega)} + \frac{1}{9c_0} \norm{\theta}^4_{L^4(\partial\Omega)} + \frac{2}{3c_0} \norm{w_h \theta}^2_{L^2(\partial \Omega)}\Big)^{1/2}
	$$
	and
	\begin{align*}
		F(\eta,w_h)(\theta) &:= \Bigg(\big(\eta (\nabla w_h\cdot{\bf{1}}),\theta \big)
		+\big(\eta (\nabla \theta\cdot{\bf{1}}),w_h \big)
		-\big(\eta (\nabla w\cdot{\bf{1}}),\theta \big)\\
		&\qquad-\sum_{i=1}^{2}\int_{\partial\Omega}\eta w_h\nu_i\theta\;d\Gamma -(c_0+2w_d)<\eta,\theta>_{\partial\Omega}\Bigg)
	\end{align*}
	as
	\begin{align*}
	\frac{d}{dt}\Bigg(e^{2\alpha t} \||\theta(t) \||^2 \Bigg)
	&+e^{2\alpha t}\norm{\theta_t}^2  
	\leq  2\alpha e^{2\alpha t}  \||\theta(t)\||^2 -\frac{d}{dt} \Big(e^{2\alpha t} \;F(\eta,w_h)(\theta)\Big) + 2\alpha e^{2\alpha t} \;F(\eta,w_h)(\theta)\\
	&+C\;e^{2\alpha t} \Big(\norm{\eta}^2+\norm{\eta_t}^2 +\norm{\eta}^2_{L^{4}(\partial\Omega)} + \norm{\eta}^2_{H^{-1/2}(\partial\Omega)}
	+ \norm{\eta_t}^2_{H^{-1/2}(\partial\Omega)}\Big)\\
	&\quad+Ce^{2\alpha t}\norm{\theta}^2 \Big(1+ \norm{w}^2_2+\norm{w_t}^2_2 +\norm{\eta_t}^2_{L^{4}(\partial\Omega)} \Big)
	+ C\;e^{2\alpha t}\norm{\theta}^4_{L^4(\partial\Omega)}\;\norm{w_{t}}_2^2 \\ 
	&\quad + C\;e^{2\alpha t} \norm{\nabla \theta}^2 \Big( \norm{w}_2^2+ \norm{w_t}_2^2 +  \norm{\eta_t}^2_{L^4(\partial\Omega)}\Big)+Ce^{2\alpha t}\norm{\nabla\theta}^2.
	\end{align*}
	Integrate the above inequality from $0$ to $t$.  Note that using  the bound $\norm{w_h}_{L^{\infty}} \leq C \;\norm{w}_2$, we obtain 
	$$2\alpha \int_{0}^{t} e^{2\alpha s}  \||\theta(s)\||^2\;ds \leq C\,\int_{0}^{t} e^{2\alpha s}  \Big(\norm{\nabla\theta}^2 + \norm{\theta}^2_{L^2(\partial\Omega)} 
	+ \norm{\theta}^4_{L^4(\partial\Omega)}\Big)\;ds$$
	and using the splitting $w_h=\theta + \tilde{w}_h$  for  first term of $F$, we  bound  $F$ as 
	\begin{align*}
	F(\eta,w_h)(\theta) & \leq C \;\Big(\norm{\eta}_{L^4(\Omega)}\;\norm{\theta}+\norm{\eta}\;\norm{w}_2 + \norm{w}_2\;\norm{\eta}_{H^{-1/2}(\partial\Omega)}\Big)
	\;(\norm{\theta} +\norm{\nabla\theta})\\
	&\leq  \frac{\nu}{2}\norm{\nabla \theta(t)}^2 + C \Big( (1+\norm{w}_2^2 )\;\norm{\theta}^2 +\norm{\eta}^2\;\norm{w}_2^2 + \norm{w}_2^2\;\norm{\eta}_{H^{-1/2}(\partial\Omega)}\Big).
	\end{align*}
	 Then 
	 use kickback argument  and    apply estimate \eqref{feq3.2}, Lemmas \ref{flm2}, \ref{flm4}, \ref{flm5}, \ref{flm3.6} and Theorem \ref{fthm3.1} to arrive at
	\begin{align*}
	e^{2\alpha t} \Big(\frac{\nu}{2} \norm{\nabla\theta(t)}^2&+ 2(c_0+ w_d)\norm{\theta(t)}^2_{L^2(\partial\Omega)}+\frac{1}{9c_0}\norm{\theta(t)}^2_{L^4(\partial\Omega)}+\frac{2}{3c_0}\int_{\partial\Omega}w_h(t)^2\theta(t)^2\;d\Gamma\Big)\\
	&+ \int_{0}^{t}e^{2\alpha s}\norm{\theta_t(s)}^2\;ds
	\leq Ch^4 \;\Big(\norm{w_0}_3\Big)\exp\Big(C\norm{w_0}_2\Big) \\
	&\qquad+ C \int_{0}^{t} e^{2\alpha s} \;\norm{w_t}^2_2 \;e^{2\alpha s} \norm{\nabla \theta(t)}^2\;ds.
	\end{align*}
	Now a use of Gronwall's Lemma with a multiplication of $e^{-2\alpha t}$ completes the rest of the proof.
\end{proof}
As a consequence  of Theorem  \ref{fthm3.2}, we obtain a super convergence result 
$\norm{\nabla (w_h(t)-\tilde w_h(t))}$. This, in turn, provides an optimal 
  order of convergence result for feedback control law. Finally, the main theorem of this section, which provides optimal error estimates (optimality with respect to approximation property) in the state variable  as well as  feedback control law is given below.
\begin{theorem}\label{fthm3.3}
There is a positive constant $C=C(\norm{w_0}_3)$ independent of $h$ such that
	\begin{align*}
	\norm{w-w_h}_{L^\infty(H^i)}\leq Ch^{2-i}e^{-\alpha t}\exp\Big(\norm{w_0}_2\Big),\quad i=0,\hspace{0.1cm}1
	\end{align*}
	and
	\begin{align*}
	\norm{v_{2t}-v_{2ht}}_{L^\infty(L^2(\partial\Omega))}\leq Ch^{3/2}e^{-\alpha t}\exp\Big(\norm{w_0}_2\Big).
	\end{align*}
\end{theorem}
\begin{proof}
	First part of the proof follows from estimates of $\eta$ in \eqref{feq3.2} and Theorems \ref{fthm3.1} and \ref{fthm3.2} with a use of triangle inequality.
	
	For the second part, we note that
	\begin{align*}
	v_{2t}-v_{2ht}=-\frac{1}{\nu}\Big((c_0+2w_d)(\eta-\theta)+\frac{2}{9c_0}(\eta-\theta)(w^2+ww_h+w_h^2)\Big).
	\end{align*}
	Hence, using $w_h=\theta+\tilde w_h$ we get
	\begin{align*}
	&\norm{v_{2t}-v_{2ht}}_{L^\infty(L^2(\partial\Omega))}\\
	&\leq C\Big(\norm{\eta}_{L^\infty(L^2(\partial\Omega))}+\norm{\theta}_{L^\infty(L^2(\partial\Omega))}\Big)\Big(1+\norm{w}^2_{L^\infty(L^4(\partial\Omega))}+\norm{\tilde w_h}^2_{L^\infty(L^4(\partial\Omega))}+\norm{\theta}^2_{L^\infty(L^4(\partial\Omega))}\Big).
	\end{align*}
	A use of Lemmas \ref{flm2}, \ref{flm3.6} and Theorem \ref{fthm3.2} completes the proof.
\end{proof}
\section{Numerical experiments}
In this section, we conduct several numerical experiments to observe stabilizability of the system \eqref{feq1.7}-\eqref{feq1.9}. More precisely, the convergence of the unsteady solution to its constant steady state solution using nonlinear Neumann feedback control law are shown. 
Moreover, we show the order of convergence  for both state variable and feedback control law by  solving \eqref{feq1.10}.  Finally, our last example concerns with the stabilization of solution for the forced viscous Burgers' equation applying linear control law in which case steady state solution is nonconstant.

For a complete discrete scheme,  we use a semi-implicit Characteristic-Galerkin method as follows:
Let 
$0<k<1$ denote the time step size and  
$t_n=nk,$ where $n$ is nonnegative integer. For smooth function $\phi$ defined on $[0,\infty),$
set $\phi ^n=\phi(t_n).$  
We now apply the  Characteristic-Galerkin method to approximate 
 $$(w_t,\varphi_h)+\Big(w\big(\nabla w\cdot {\bf{1}}\big),\varphi_h\Big)\approx (\frac{W^{n}-W^{n-1}(X^{n-1}(x))}{k},\varphi_h),$$ where $X^{n-1}(x)$ is an approximation of the solution at $t=(n-1)k$ of the ordinary differential equation $\frac{d{\bf X}(t)}{dt}={\bf W}^{n-1}({\bf X}(t))$, ${\bf X}(nk)=x$, where ${\bf W}^{n-1}(x)=(w(x,(n-1)k), w(x,(n-1)k))$ and ${\bf X}=(X_1,X_2)$. Now this can be easily solve by Freefem++ 'convect' operator command. Hence, finally we seek  $\{{W^n}\}_{n\geq 1}\in V_h$ as a solution of
 \begin{align}
 \frac{(W^n,\varphi_h)}{k}&-\frac{1}{k}\Big((convect([W^{n-1}, W^{n-1}],-k,W^{n-1}),\varphi_h)\Big)+\nu(\nabla W^n,\nabla\varphi_{h})+w_d\big(\nabla W^n\cdot{\bf {1}},\varphi_h\big)\notag\\
 &+\Big\langle 2(c_0+w_d)W^n+(\frac{2}{9c_0}(W^{n-1})^2)W^n,\varphi_h\Big\rangle=0  \quad \forall~ \varphi_h \in V_h\label{feqn5.2}, 
 \end{align}
 with $W^0=w_{0h}$.
 For more details see \cite{Hecht}.
 For state and control trajectories final plot, we use Matlab.
\begin{example}\label{ex4.1}
We choose initial condition $w_0$ as $w_0=x_1(x_1-1)x_2(x_2-1)-3,$ where $w_d=3$ is the steady state solution and  $\nu=1$ with $\Omega=[0,1]\times [0,1]$. For uncontrolled solution, we take zero Neumann boundary condition in \eqref{feq1.10} and corresponding solution is denoted as "Uncontrolled solution" in Figure \ref{fig:6.1}. For controlled solution, we choose the Neumann control \eqref{feqx1} with $c_0=1$ and corresponding solution is denoted as "Controlled solution, $c_0=1$" in Figure \ref{fig:6.1}.
\end{example}
\begin{figure}[ht!]
	\begin{minipage}[b]{.5\linewidth}
		\centering
		
		\includegraphics[height=6cm]{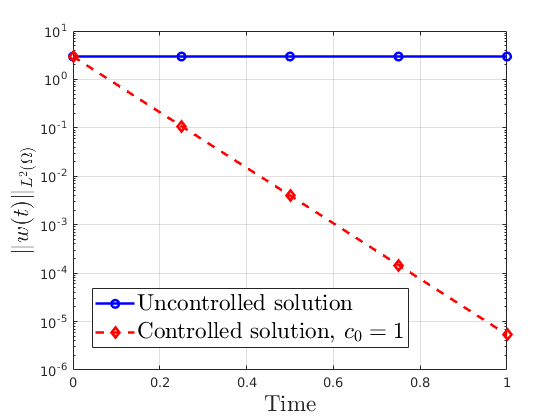}
		\caption{State, Example \ref{ex4.1}}
		\label{fig:6.1}
	\end{minipage}
	\hspace{0.05cm} 
	\begin{minipage}[b]{0.5\linewidth}
		\centering
		\includegraphics[height=6cm]{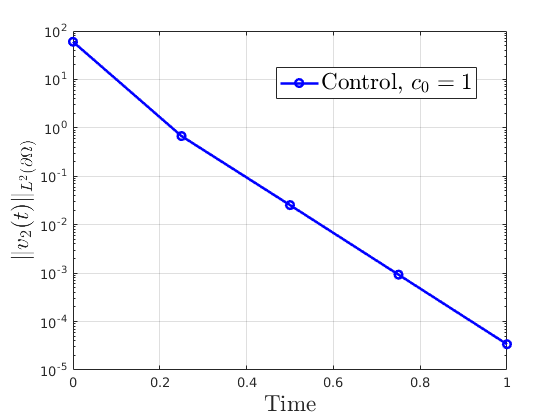}
		\caption{Control, Example \ref{ex4.1}}
		\label{fig:6.2}
	\end{minipage}
\end{figure}	
\begin{table}[ht]
	\centering
	\caption{Errors and convergence rate of $w$ when $c_0=1$, $k=0.0001$ and $t=1$ for Example \ref{ex4.1}}
	\begin{tabular}{c c c c c}
		\hline
		$h$ & $\norm{w(t_n)-W^n}$ & Conv. Rate & $\norm{w(t_n)-W^n}_1$ & Conv. Rate
		\\  \hline \hline
		$\frac{1}{4}$ & $1.26\times 10^{-7}$  &          & $5.82\times 10^{-7}$ & \\ \hline
		$\frac{1}{8}$ & $3.82\times 10^{-8}$ & 1.72 & $2.43\times 10^{-7}$ & 1.26 \\ \hline
		$\frac{1}{16}$ & $9.78\times 10^{-9}$ & 1.96 & $1.07\times 10^{-7}$ & 1.18 \\ \hline
		$\frac{1}{32}$ & $2.44\times 10^{-9}$ & 2.01 & $4.88\times 10^{-8}$ & 1.12 \\ \hline
		$\frac{1}{64}$ & $5.99\times 10^{-10}$ & 2.03 & $2.13\times 10^{-8}$ & 1.19\\ \hline
		\label{table:6.1}
	\end{tabular}
\end{table}
\begin{table}[ht]
	\centering
	\caption{Errors and convergence rate of $v_2$ when $c_0=1$ and $t=1$ for Example \ref{ex4.1}}
	\begin{tabular}{c c c c c}
		\hline
		$h$ & $|{v_2(t_n)-v_{2h}(t_n)}|$ & Conv. Rate 
		\\  \hline \hline
		$\frac{1}{4}$ & $8.75\times 10^{-7}$ &           \\ \hline
		$\frac{1}{8}$ & $2.61\times 10^{-7}$ & 1.74  \\ \hline
		$\frac{1}{16}$ & $6.67\times 10^{-8}$ & 1.97   \\ \hline
		$\frac{1}{32}$ & $1.65\times 10^{-8}$ & 2.01  \\ \hline
		$\frac{1}{64}$ & $4.03\times 10^{-9}$ & 2.04 \\ \hline 
		\label{table:6.2}
	\end{tabular}
\end{table}
 From Figure  \ref{fig:6.1}, we can easily see that without any control i.e. with zero Neumann boundary, solution of \eqref{feq1.10} does not settle at zero, whereas applying the control \eqref{feqx1}, the solution for the problem \eqref{feq1.10} in $L^2$- norm goes to zero. Also it is observed that for
 other values of $c_0>0$, the system \eqref{feq1.10} is stabilizable. From Table \ref{table:6.1}, it follows that $L^2$ and $H^1$ orders of convergence
 for state variable $w(t)$ are $2$ and $1$, respectively, which confirms our theoretical results established in Theorem \ref{fthm3.3}. 
 Since the exact solution is unknown in this case, we have taken very refined mesh solution as exact solution to compute the order of convergence.
 In Table \ref{table:6.2}, it is noted that the order of convergence of feedback control law \eqref{feqx1} is $2$, while  theoretically it is proved to be $3/2$ in Theorem \ref{fthm3.3}. 
\begin{example}\label{ex4.2}
In this example, take the initial condition $w_0=\sin(\pi x_1)\sin(\pi x_2)$ and  $\nu=0.05$, $c_0=1$ with $0$ as the steady state solution in $\Omega=[0,1]\times [0,1]$.
\end{example}
\begin{figure}[ht!]
	\begin{minipage}[b]{0.5\linewidth}
		\centering
		
		\includegraphics[height=6cm]{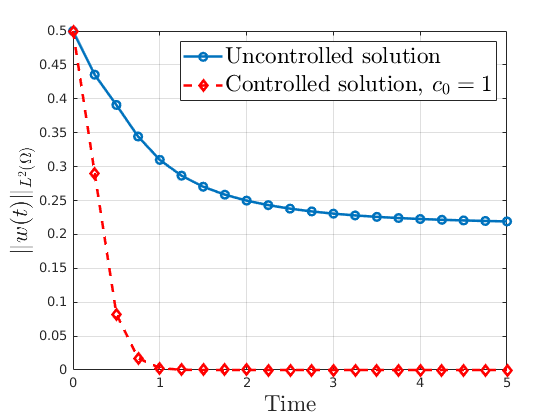}
		\caption{State, Example \ref{ex4.2}}
		\label{fig:6.3}
	\end{minipage}
	\hspace{0.05cm} 
	\begin{minipage}[b]{0.5\linewidth}
		\centering
		\includegraphics[height=6cm]{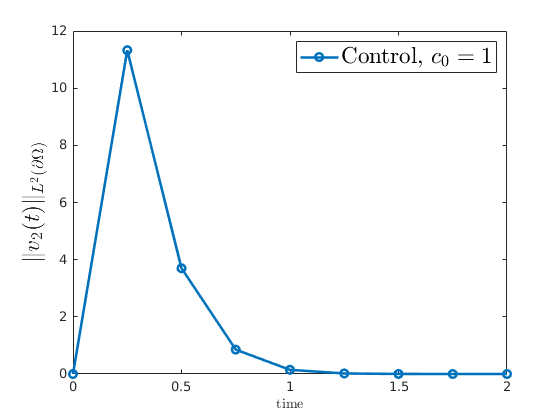}
		\caption{Control, Example \ref{ex4.2}}
		\label{fig:6.4}
	\end{minipage}
\end{figure}
From Figure \ref{fig:6.3}, it is observed that steady state solution $w_d=0$ is unstable in the first case denoted as "Uncontrolled solution ". But  using the control law \eqref{feqx1}, it is shown that state $w$ in $L^2$- norm goes to zero  exponentially. Figure \ref{fig:6.4} indicates how  control law \eqref{feqx1} behave with time and tends to settle at zero after some time.\\
Now the following example is related to the Remark \ref{rm} where the control law is applied to some part of the boundary. In the remaining part, either zero Dirichlet or zero Neumann condition is considered.
\begin{example}\label{ex3}
We take the initial condition $w_0=\cos(\pi x_1)\cos(\pi x_2)-5$, where $w_d=5$ is the steady state solution with  $\nu=0.01$ and $\Omega=[0,1]\times [0,1]$. We consider two cases.\\
{\bf Case 1:} Take zero Dirichlet boundary $\Gamma_D={1}\times[0,1]$ and on the remaining parts Neumann boundary control $\Gamma_N$. For controlled solution, we take \eqref{feqx1} with $c_0=10$ on $\Gamma_N$ i.e. on 3 parts of the boundary. We then compare the trajectories with another case which is given below.\\
{\bf Case 2:}  We consider zero Dirichlet boundary $\Gamma_D={1}\times[0,1]$, put the Neumann boundary control on ${0}\times[0,1]$ and  put zero Neumann boundary condition on other 2 parts. So, in Case 2 control works only on one part of the boundary.\\
For uncontrolled solution, we take zero Neumann boundary condition on $\Gamma_N$ and zero Dirichlet on $\Gamma_D={1}\times[0,1]$. 
\end{example}
\begin{figure}[ht!]
	\begin{minipage}[b]{0.5\linewidth}
		\centering
		
		\includegraphics[height=6cm]{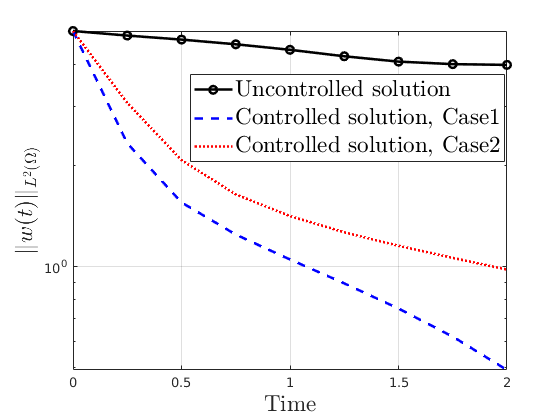}
		\caption{State, Example \ref{ex3}}
		\label{fig:x6.3}
	\end{minipage}
	\hspace{0.05cm} 
	\begin{minipage}[b]{0.5\linewidth}
		\centering
		\includegraphics[height=6cm]{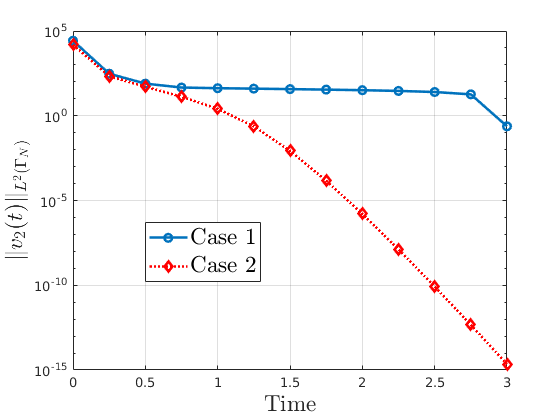}
		\caption{Control, Example \ref{ex3}}
		\label{fig:x6.4}
	\end{minipage}
\end{figure}
From Figure \ref{fig:x6.3}, it is clear that with homogeneous mixed boundary condition, uncontrolled solution does not change its state to zero whereas in presence of control on $\Gamma_N$ only, state and control trajectories go to zero both in Case 1 and Case 2.  Since in Case 2, control works only in one part of the boundary, so control needs more time compare to Case 1 to settle its state to zero, which is also visible in Figure \ref{fig:x6.3}. The $L^2$- norm of feedback control has a tendency to settle at zero faster in the Case 2 compare to Case 1 which is documented in Figure \ref{fig:x6.4}. Also we have not observed much differences in the trajectories if we change the zero Dirichlet by zero Neumann condition in Case 1. Concerning zero Dirichlet boundary on more parts of the boundary e.g. if we take $w=0$ on 3 parts of the boundary with Neumann control on remaining one part, corresponding state tends to settle at zero even within $t<1$ since with zero Dirichlet boundary, the system is already stable.\\

Below we discuss another example, where steady state solution is not a constant for the forced Burgers' equation. It will be shown that, even with the linear control law, system can be stabilizable numerically. 
\begin{example}\label{ex4}
We now consider a case when the steady state solution is not constant:
\begin{align}
&u_t-\nu\Delta u+u(\nabla u\cdot {\bf{1}})=f^\infty\qquad\text{in}\quad (x,t)\in \Omega\times(0,\infty)\label{feq6.1},\\
&\frac{\partial u}{\partial n}(x,t)=g^\infty+v_2(x,t)\qquad \text{on} \quad(x,t)\in \partial \Omega \times (0,\infty)\notag,\\
&u(x,0)=u_0(x)\qquad x\in \Omega\notag,
\end{align}
where, $f^\infty$ and $g^\infty$, independent of $t$  are functions of $x_1$ and $x_2$ only.
Corresponding equilibrium or steady state solution $u^\infty$ of the unsteady state problem satisfies
\begin{align}
-\nu \Delta u^\infty+u^\infty(\nabla u^\infty\cdot {\bf{1}})&=f^\infty \qquad\text{in} \quad \Omega \label{feq6.2},\\
\frac{\partial u^\infty}{\partial n}&=g^\infty \quad \text{on} \quad \partial \Omega \notag.
\end{align}
Let $w=u-u^\infty$. Then, $w$ satisfies
\begin{align}
&w_t-\nu \Delta w+u^\infty(\nabla w\cdot {\bf{1}})+w(\nabla u^\infty\cdot {\bf{1}})+w(\nabla w\cdot {\bf{1}})=0 \qquad\text{in}\quad (x,t)\in \Omega\times(0,\infty)\label{feq6.3},\\
&\frac{\partial w}{\partial n}(.,t)=v_2(x,t),\quad \text{on} \quad \partial \Omega\times(0,\infty)\notag,\\
&w(0)=u_0-u^\infty=w_0(\text{say})\quad\text{in}\quad\Omega \notag.
\end{align}
 Similarly as before we solve \eqref{feq6.3} using Freefem++ as in \eqref{feqn5.2} with linear control law $v_2=-\frac{1}{\nu}c_0W^n$.\\
For the numerical experiment, we choose viscosity parameter $\nu=0.1$, steady state solution $u^\infty=-0.2x_1$, forcing function $f^\infty=0.04x_1$ and $g^\infty= -0.2n_1$, control parameter $c_0=10$ with initial condition $w_0=\sin(\pi x_1)\sin(\pi x_2)+0.2x_1$ in $\Omega=[0,1]\times [0,1]$.
\end{example}
\begin{figure}[ht!]
	\begin{minipage}[b]{.5\linewidth}
		\centering
		
		\includegraphics[height=6cm]{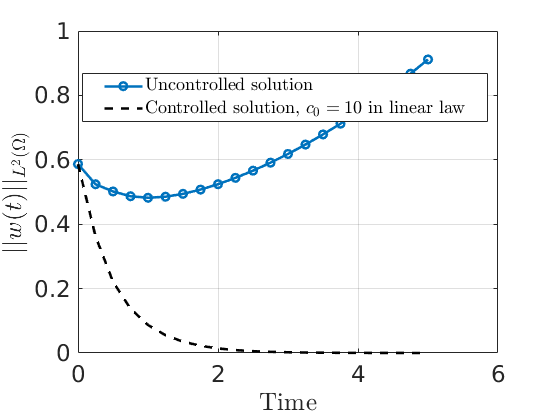}
		\caption{State, Example \ref{ex4}}
		\label{fig:6.5}
	\end{minipage}
	\hspace{0.05cm} 
	\begin{minipage}[b]{0.5\linewidth}
		\centering
		\includegraphics[height=6cm]{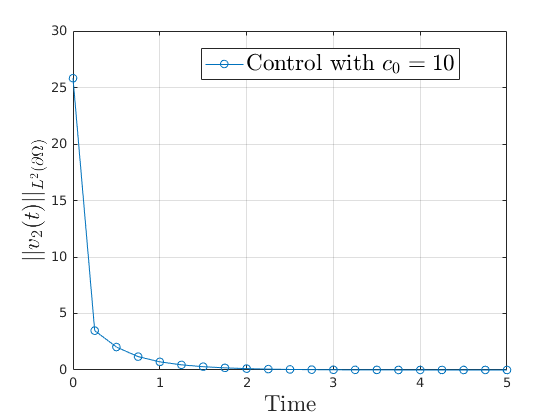}
		\caption{Control, Example \ref{ex4}}
		\label{fig:6.6}
	\end{minipage}
\end{figure}
\noindent
From the first draw line in Figure \ref{fig:6.5}, we observe that nonconstant steady state solution  is not asymptotically stable with zero Neumann boundary condition. But, using the linear control law
$-\frac{1}{\nu}c_0W^n$, system \eqref{feq6.3} is stabilizable which is also documented in Figure \ref{fig:6.5}. Figure \ref{fig:6.6} shows that the linear control law
$-\frac{1}{\nu}c_0W^n$ decays to zero as time increases.
However, we do not have a theoretical result to substantiate this observation. We believe that the system is locally stabilizable with this linear control law.

\section{Concluding Remarks.}
In this paper, global stabilization results for the two dimensional viscous Burgers'
equation are established in $L^\infty(H^i)$, $i=0,1,2$ norms, when the steady state solution is constant. 
Optimal error estimates in $L^\infty(L^2)$ and in $L^\infty(H^1)$ for the state variable are established. Further, error estimate for the feedback 
controller is also shown. All the results are verified by numerical examples.
Now under addition of forcing function in the two dimensional viscous Burgers' equation, the steady state solution is no more constant and as such the present analysis does not hold for nonconstant steady state case. Hence, the analysis for
two dimensional generalized forced viscous Burgers' equation  will be addressed in future.

{\bf{Acknowledgements}}
Both authors acknowledge the valuable suggestions and comments given by  honorable  referees  which help to improve the  manuscript. 
The first author was supported by the ERC advanced grant 668998 (OCLOC) under the EUs H2020 research program.

\bibliographystyle{amsplain}
 
\end{document}